\documentclass[a4paper,10pt]{article}

\RequirePackage{amsthm,mathtools,amsfonts}
\RequirePackage[colorlinks,citecolor=blue,urlcolor=blue,linkcolor=blue]{hyperref}
\RequirePackage{graphicx,authblk}
\usepackage[authoryear]{natbib}
\usepackage{geometry}
\usepackage{subcaption}
\usepackage{apptools}
\theoremstyle{plain}

\newtheorem{theorem}{Theorem}[section]

\newtheorem{cor}[theorem]{Corollary}
\newtheorem{lemma}[theorem]{Lemma}
\newtheorem{definition}[theorem]{Definition}
\theoremstyle{remark}
\newtheorem{remark}{Remark}
\newtheorem{example}{Example}

\newcommand{\R}{\mathbb{R}}
\newcommand{\N}{\mathbb{N}}
\newcommand{\Z}{\mathbb{Z}}

\newcommand{\bP}{\mathbb{P}}

\newcommand{\cF}{\mathcal{F}}

\newcommand{\cM}{\mathcal{M}}
\newcommand{\cD}{\mathcal{D}}
\newcommand{\cN}{\mathcal{N}}
\newcommand{\cS}{\mathcal{S}}
\newcommand{\maxt}{\beta_t}

\newcommand{\eps}{\varepsilon}

\newcommand{\BND}{Brownian normal distribution}
\newcommand{\E}{\mathbb{E}}
\newcommand{\dist}{\mathrm{dist}}
\newcommand{\hfun}{h_{t,m}}
\newcommand{\lfun}{\ell_{t,m}}
\newcommand{\tbound}{\delta}
\newcommand{\hfunl}{f_{t,m}}

\newfont{\handw}{cmmi10 scaled 1200}
\newfont{\handws}{cmmi10 scaled 800}
\newcommand{\lw}{\mbox{\handw \symbol{96}}}

\DeclareMathOperator*{\argmin}{argmin}

\newcommand\commentout[1]{}

\begin{document}
  
\title{Diffusion Means in Geometric Spaces}

\author[1]{Benjamin Eltzner\footnote{benjamin.eltzner@mpinat.mpg.de}}
\author[2]{Pernille E. H. Hansen\footnote{pehh@di.ku.dk}}
\author[3]{Stephan Huckemann\footnote{huckeman@math.uni-goettingen.de}}
\author[2]{Stefan Sommer\footnote{sommer@di.ku.dk}}
\affil[1]{\normalsize Max Planck Institute for Multidisciplinary Sciences, Göttingen, Germany.}
\affil[2]{\normalsize Department of Computer Science,	University of Copenhagen, Copenhagen, Denmark.}
\affil[3]{\normalsize Institute for Mathematical Stochastics,	Georg-August-Universität Göttingen, Göttingen, Germany.}

\date{\today}

\maketitle

\begin{abstract}
	We introduce a location statistic for distributions on non-linear geometric spaces, the diffusion mean, serving as an extension and an alternative to the Fr\'echet mean. The diffusion mean arises as the generalization of Gaussian maximum likelihood analysis to non-linear spaces by maximizing the likelihood of a Brownian motion. 
	The diffusion mean depends on a time parameter $t$, which admits the interpretation of the allowed variance of the diffusion. The diffusion $t$-mean of a distribution $X$ is the most likely origin of a Brownian motion at time $t$, given the end-point distribution $X$. 
We give a detailed description of the asymptotic behavior of the diffusion estimator and provide sufficient conditions for the diffusion estimator to be strongly consistent. Particularly, we present a smeary central limit theorem for diffusion means and we show that joint estimation of the mean and diffusion variance rules out smeariness in all directions simultaneously in general situations.
Furthermore, we investigate properties of the diffusion mean for distributions on the sphere $\cS^n$.
Experimentally, we consider simulated data and data from magnetic pole reversals, all indicating similar or improved convergence rate compared to the Fr\'echet mean. Here, we additionally estimate $t$ and consider its effects on smeariness and uniqueness of the diffusion mean for distributions on the sphere. 
\end{abstract}


\section{Introduction}

In many systems of interest, measured quantities cannot be easily represented in a Euclidean space. Important examples include directional data on a circle or sphere, cf. \citet{MJ2000} and landmark shapes, cf. \citet{S1996,DM16}. For a statistical treatment of such data, it is necessary to generalize statistical concepts from Euclidean spaces to more general spaces.

The mean, which is an important location statistic, was generalized to metric spaces by Fr\'echet \citet{frechet_les_1948} whose generalization has had a lasting impact on the statistics literature. For this reason, in statistics, this mean is called the \emph{Fr\'echet mean} 
while in the field of geometry it is referred to as the \emph{barycenter}. It is defined for random variables $X$ on metric spaces $(Q,d)$ as minimizers of the Fr\'echet function, i.e. the set of points $E(X) = \argmin_{y\in Q}\E[d(X,y)^2]$. A common example is the \emph{intrinsic mean}, also called the \emph{Riemann center of mass} for a smooth manifold $Q$ with geodesic distance $d$, leading to many important results regarding uniqueness \citet{karcher_riemannian_1977,kendall_probability_1990,Gr05,Afsari11},
strong consistency of estimators \citet{ziezold_expected_1977,bhattacharya_large_2003,huckemann_intrinsic_2011,HuckemannEltzner2018,schotz_strong_2020,evans_strong_2020}
and central limit theorems (CLTs) \citet{bhattacharya_large_2005,renlund_limit_2011,huckemann_inference_2011,eltzner_smeary_2018}. 

While we will briefly touch on other types of Fr\'echet means below, to keep language simple, in this entire article, the standalone term \emph{Fr\'echet mean} always refers to the intrinsic mean.

While the Fr\'echet mean reduces to the usual mean vector on Euclidean spaces, it is by far not the only generalization that satisfies this property. A widely used alternative mean concept exists on embedded submanifolds $\cM \subset \mathbb{R}^k$. There one can define an \emph{extrinsic mean} as the shortest distance orthogonal projection in the ambient space 
of the mean in $\mathbb{R}^k$ to $\cM$. This mean 
uses not the squared geodesic distance but the squared chordal distance, which is simply the distance in $\mathbb{R}^k$ \citet{HL96,HL98,bhattacharya_large_2003,bhattacharya_large_2005,Hotz2013}. We discuss the relation of intrinsic and extrinsic means to the proposed diffusion means below. Also relying on embeddings, residual, Ziezold and Procrustean means have been defined 
\citet{Gow,Z94,MJ2000,LeKume2000,H_meansmeans_12,DM16}.
On Euclidean spaces, embedded in a higher dimensional Euclidean space, all of theses 
reduce to the usual mean vector: the expected vector value.

The mean value concept was extended further by \citet{ziezold_expected_1977} to quasi-metric spaces and later the \emph{generalized Fr\'echet means} \citet{huckemann_intrinsic_2011,huckemann_inference_2011} 
were introduced for continuous functions $\rho:Q\times P \to [0,\infty)$ on topological spaces $Q$ and $P$, with $Q$ being the data space and $P$ the descriptor space. Generalized Fr\'echet means can be understood, 
e.g. as geodesics or nested lower dimensional subspaces 
\citet{HuckemannEltzner2018}. 
They are defined as minimizing a squared loss function $\rho$ 
linking the data space to the descriptor space
obtaining a meaningful mean set $E^{(\rho)}(X) = \argmin_{p\in P}\E[\rho(X,p)^2]$.

In this paper we introduce \emph{diffusion means} on geometric spaces and give a theoretical account for the definition on connected smooth manifolds. Diffusion means have several desirable properties. Like the Fr\'echet mean, they reduce to the usual mean vector on Euclidean spaces and are defined purely intrinsically, without reference to an embedding. In fact, they do not even require a Riemannian structure, defining isotropic Brownian motion on a manifold equipped with a reference measure suffices for the definition of diffusion means. Furthermore, they are closely tied to a measure of variance, which can be simultaneously estimated. We show that diffusion means for optimal diffusion variance, unlike the intrinsic Fr\'echet mean, only exhibit smeariness as described by \citet{eltzner_smeary_2018} in extreme corner cases and otherwise satisfy a standard CLT with rate $n^{-1/2}$. In consequence, diffusion means estimated simultaneously with the diffusion variance are more reliable than sample Fr\'echet means.

The diffusion mean value can be realized as a generalized Fr\'echet mean for $P=Q$ that gives a location statistic for random variables on $Q$. The definition is motivated by the idea of generalizing the maximum likelihood estimate (MLE) of the location parameter of the multivariate normal distribution by maximizing the likelihood of a Brownian motion instead of minimizing squared distance. The Euclidean MLE of the location parameter $\mu$ of $\cN(\mu, \Sigma)$ is independent of $\Sigma$ and coincides with the mean vector, and the family of normal distributions are generated by Brownian motions with transition densities being the heat kernel. There is no single canonical generalization of the normal distribution to manifolds, especially for anisotropic $\Sigma$. However, the isotropic case $\Sigma = t \cdot \textnormal{id}$ can be generalized whenever Brownian motions and heat kernels can be defined. Thus, we base the maximum likelihood estimation on the family heat kernels $p(x,\mu,t)$, because they are the transition densities of Brownian motions on manifolds. Similar to the Euclidean case, the unknown parameters $\mu$ can be considered as the mean value and $t$ as the variance parameter. For a fixed $t>0$, we choose the heat kernel $p$ as our loss function and define the \emph{diffusion $t$-mean set} as the minimizers 
\begin{equation*}
E_t(X) = \argmin_{y\in Q}\E[-\ln p(X,y,t)].
\end{equation*} 
The definition is not limited to spaces with a smooth structure and we mention how the definition can be extended to random variables on graphs by only using the existence of random walks. However, in contrast to Euclidean space, the estimated valued of $\mu$ in general depends on $t$ and uniqueness is not always guaranteed.

In Euclidean space, the variance $t$ does not influence the estimated mean. While diffusion means on manifolds typically depend on $t$, this does not mean that diffusion means for larger $t$ are less dependent on the data. Instead, with increasing $t$ the diffusion process and thus the diffusion mean is influenced by the manifold's geometry on an increasingly larger scale.

The logarithm of the heat kernel is naturally connected to the squared geodesic distance by the limit $\lim_{t \to 0} -2t\ln p(x,y,t) = \dist(x,y)^2$ \citet{hsu_stochastic_2002}, which allows us to consider diffusion means as an extension of Fr\'echet means. We can interpret $t$ as the amount of variance connected to the diffusion mean, where the Fr\'echet mean represents the limit of vanishing diffusion variance. Throughout the paper we consider the similarities and differences between diffusion means and the Fr\'echet mean. To exemplify this, we compare in Figure \ref{BrownianNormalDis} the estimated variance of the Fr\'echet and diffusion mean on different bimodal \BND s. 

\begin{figure}[ht]
	\begin{subfigure}[t]{.45\linewidth}
		\centering
		\includegraphics[height = 6cm,trim={ 3cm 1cm 1cm 1cm},clip]{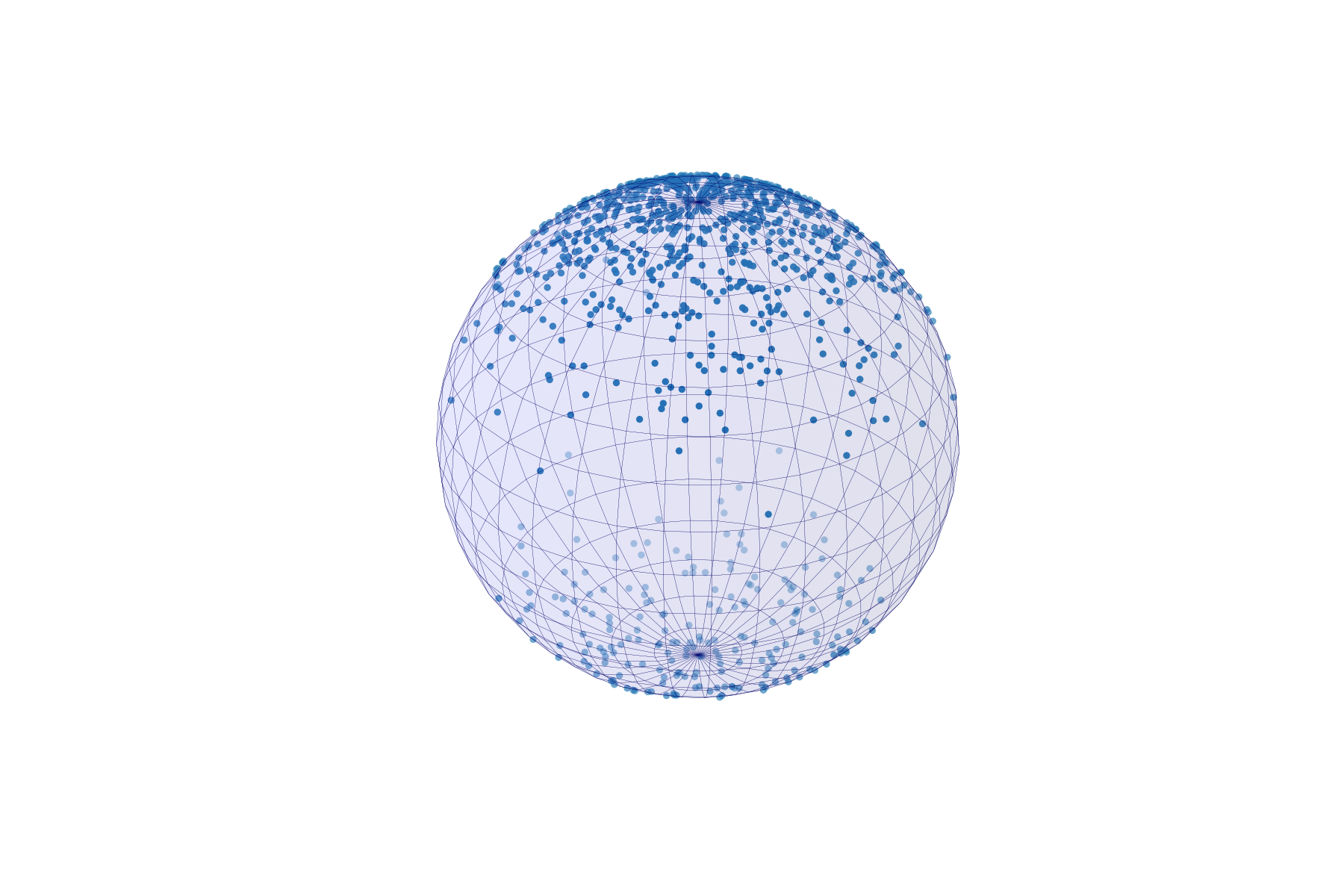}
		\caption{Sample from a bimodal \BND \space on $\cS^2$ with variance $\sigma^2 = 0.3$ at both poles and with weight $0.8$ at the north pole and $0.2$ at the south pole, see Example \ref{ex:bimodal-brownian}.}
		\label{BNDdistribution}
	\end{subfigure}
	\hspace{1em}
	\begin{subfigure}[t]{.45\linewidth}
		\centering
		\includegraphics[height = 6cm, trim={0 0cm 0 0cm,clip}]{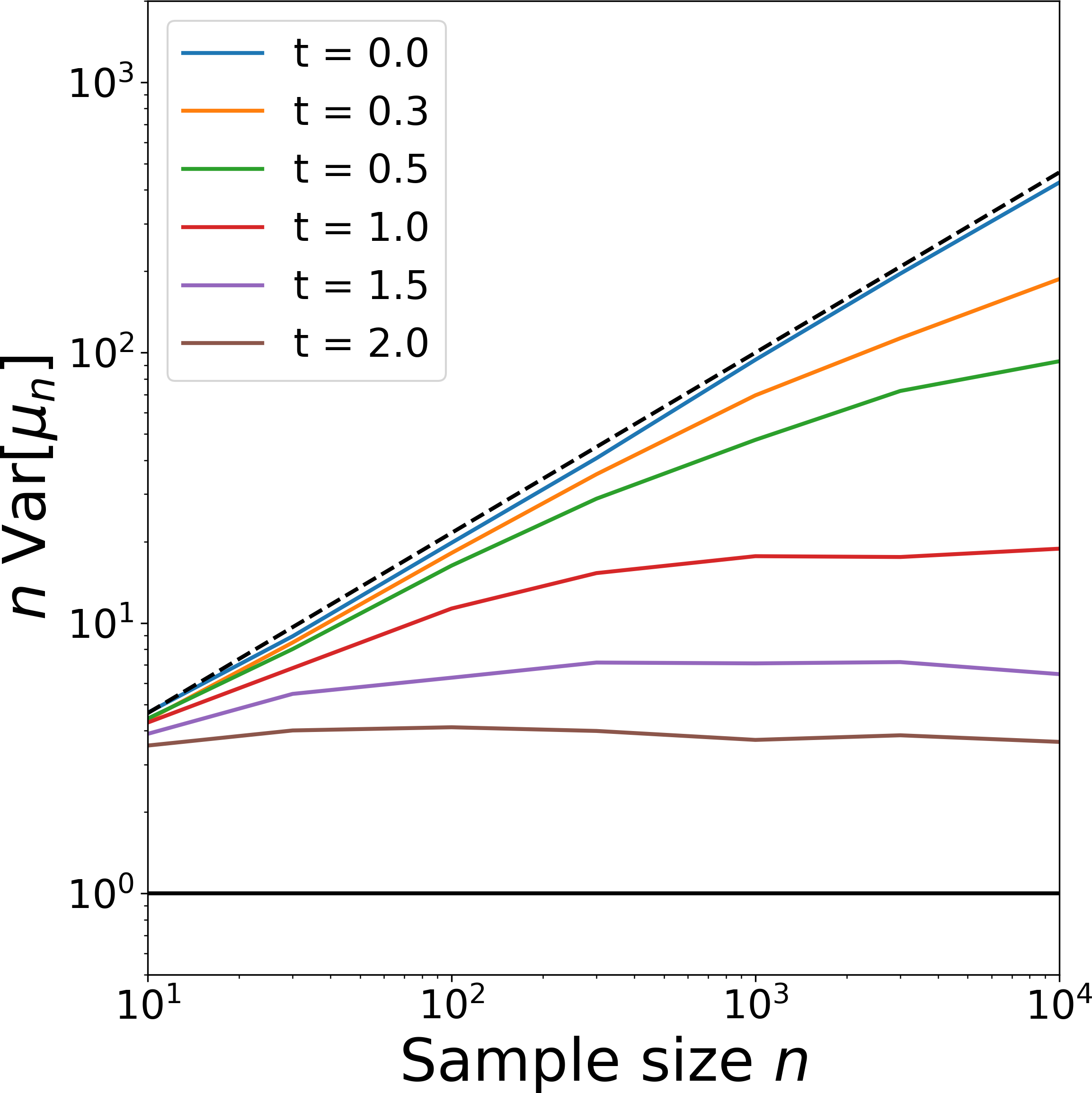}
		\caption{Variances of
			diffusion means scaled with sample size $n$ for $t = 0.3,0.5,1,1.5,2$ and of the Fr\'echet mean ($t=0$) for the distribution in Figure 1~(a). Horizontals indicate the usual CLT rate of $1/\sqrt{n}$ , the black horizontal indicates the standard CLT under absence of finite sample smeariness and the dashed line indicates $2$-smeariness.}
		\label{BNDestimates}
	\end{subfigure}
	\caption{Panel~(b) compares the convergence rate of the diffusion means for $t = 0.4,0.6,1,2,4$ and the Fr\'echet means on the \BND\ from Panel~(a). With increasing $t$, the blue curve (Fr\'echet mean) follows $2$-smeariness, whereas the curves for the diffusion means approach the usual CLT rate.}
	\label{BrownianNormalDis}
\end{figure}
By sampling from a Brownian motion on $\cS^2$ with variance $\sigma>0$ starting at the north pole $\mu$ and south pole $-\mu$ respectively, we obtain samples from two \BND s. Combining the two, we get a bimodal \BND\ by sampling from the northern distribution with probability $1-\alpha$ and from the southern with probability $\alpha$. The plots in Figure~\ref{BrownianNormalDis} show the result of estimating the Fr\'echet mean and the diffusion mean for $t = 0.4,0.6,1,2,4$ when setting $\sigma = 0.3$ and $\alpha = 0.2$. The Fr\'echet sample mean is unique at $\mu$, but the estimator appears to converge to the population mean more slowly than with the standard CLT rate $1/\sqrt{n}$. 

This slower rate has been termed \emph{smeariness} by \citet{HuckemannEltzner2018} and among others, it lowers the power of asymptotic tests.
We see a clear improvement in the convergence rate for the diffusion means: Their curves approach a horizontal above the usual CLT rate. This phenomenon has been coined \emph{finite sample smeariness} (FSS) by \citet{hundrieser_finite_2020}, which, among others, still negatively affects powers asymptotic test. Clearly, the increase of FSS of diffusion means wanes with increasing $t$.

Contrary to the Fr\'echet mean, but similar to the extrinsic mean, a unique diffusion mean admits a positive probability mass on the cut locus. In \citet{hotz_intrinsic_2015}, it was shown that for distributions on the circle $\cS^1$, a unique Fr\'echet mean implies no point mass on the cut locus of the mean, and this result has later been extended further to distributions on complete and connected manifolds \citet{le_measure_2014}. On the sphere $\cS^m$ of dimension $m\geq 2$, this means that a point $\mu$ cannot be the Fr\'echet mean of a distribution with nonzero probability mass in the point $-\mu$. 
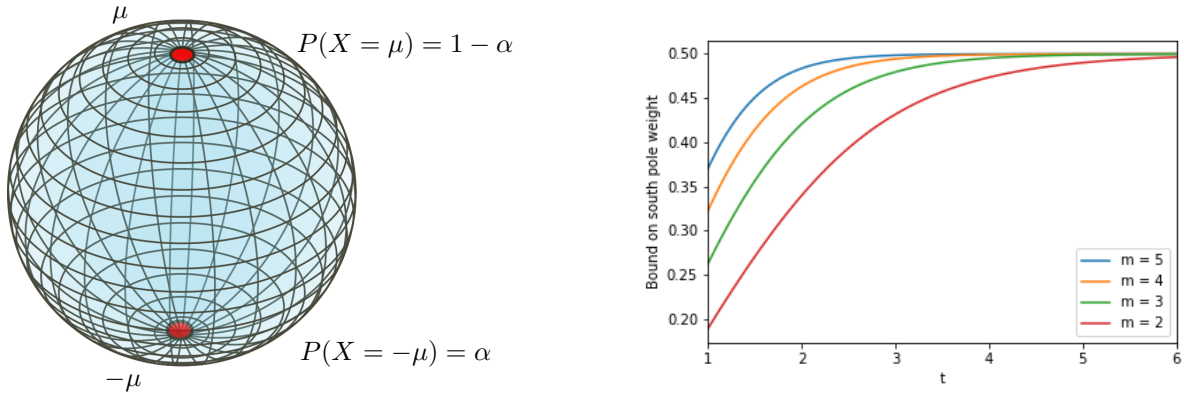
\begin{figure}[ht]
	\centering
	\begin{subfigure}[t]{.48\linewidth}
		\def\svgwidth{1\linewidth}
\begingroup%
  \makeatletter%
  \providecommand\color[2][]{%
    \errmessage{(Inkscape) Color is used for the text in Inkscape, but the package 'color.sty' is not loaded}%
    \renewcommand\color[2][]{}%
  }%
  \providecommand\transparent[1]{%
    \errmessage{(Inkscape) Transparency is used (non-zero) for the text in Inkscape, but the package 'transparent.sty' is not loaded}%
    \renewcommand\transparent[1]{}%
  }%
  \providecommand\rotatebox[2]{#2}%
  \newcommand*\fsize{\dimexpr\f@size pt\relax}%
  \newcommand*\lineheight[1]{\fontsize{\fsize}{#1\fsize}\selectfont}%
  \ifx\svgwidth\undefined%
    \setlength{\unitlength}{260.84391346bp}%
    \ifx\svgscale\undefined%
      \relax%
    \else%
      \setlength{\unitlength}{\unitlength * \real{\svgscale}}%
    \fi%
  \else%
    \setlength{\unitlength}{\svgwidth}%
  \fi%
  \global\let\svgwidth\undefined%
  \global\let\svgscale\undefined%
  \makeatother%
  \begin{picture}(1,0.72155752)%
    \lineheight{1}%
    \setlength\tabcolsep{0pt}%
    \put(0,0){\includegraphics[width=\unitlength]{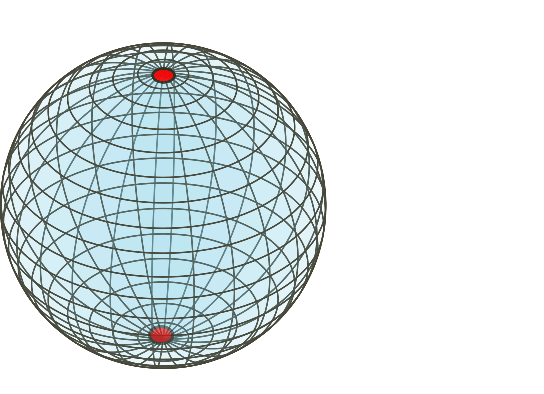}}%
    \put(0.5047638,0.05709254){\color[rgb]{0,0,0}\makebox(0,0)[lt]{\lineheight{1.25}\smash{\begin{tabular}[t]{l}$P(X = -\mu) =\alpha$\end{tabular}}}}%
    \put(0.49913394,0.59125284){\color[rgb]{0,0,0}\makebox(0,0)[lt]{\lineheight{1.25}\smash{\begin{tabular}[t]{l}$P(X = \mu) = 1-\alpha$\end{tabular}}}}%
    \put(0.18128057,0.65037018){\color[rgb]{0,0,0}\makebox(0,0)[lt]{\lineheight{1.25}\smash{\begin{tabular}[t]{l}$\mu$\end{tabular}}}}%
    \put(0.16803139,0.01172489){\color[rgb]{0,0,0}\makebox(0,0)[lt]{\lineheight{1.25}\smash{\begin{tabular}[t]{l}$-\mu$\end{tabular}}}}%
  \end{picture}%
\endgroup%

		\caption{The two pole distribution on $\cS^2$ concentrated at north pole $\mu$ (red) and south pole $-\mu$ (red) \\ with $P(X = \mu)= 1-\alpha$ and $P(X =-\mu)=\alpha$.} 
		\label{2poledis}
	\end{subfigure}\hfill
	\begin{subfigure}[t]{.5\linewidth}
		\centering
		\includegraphics[height =5.3cm, trim={0.3cm 0 1.3cm 0},clip]{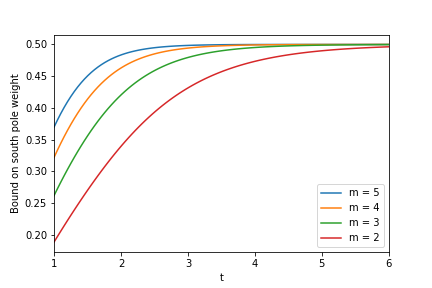}
		\caption{The upper bound $\Lambda_m(t)$ for which the diffusion $t$-mean of the 2 pole distribution is $\mu$ for $m=2,3,4,5$ and $t>\delta(m)$.}
		\label{2polesmeary}
		\label{introfig}
	\end{subfigure}\hfill
	\caption{The Fr\'echet mean of the two pole distribution on $\cS^m$, detailed in Example \ref{ex:two_points1}, displayed in (a) for $m=2$ does not exist 
	for any $\alpha>0$ as the Riemann center of mass is a latitude ring. For $t>\delta(m)$, the diffusion $t$-mean is $\mu$ whenever $\alpha$ is less than the upper bound $\Lambda_m(t)$ displayed in (b), converging to $\frac12$ for all $m\geq 2$.}
	\label{IntroFig}
\end{figure}

We show that the diffusion means are compatible with positive mass on the cut locus by considering the spherical distribution $P(X = \mu)= 1-\alpha$ and $P(X =-\mu)=\alpha$ in Figure \ref{IntroFig}~(a), which we will refer to as the \emph{two pole distribution}. For $t>1.1$, we show that the diffusion $t$-mean is $\mu$ whenever $\alpha\leq \Lambda_m(t)$ for a bound $\Lambda_m(t)$ satisfying $\lim_{t \to \infty}\Lambda_m(t)\to \frac{1}{2}$ for all $m\geq 2$.

From a practical viewpoint, a response to the properties such as smeariness of the Fr\'echet mean is often to use the extrinsic mean. However, the extrinsic mean is per definition dependent on the choice of embedding to which there is often no canonical choice. Compared to this, the diffusion mean being an intrinsic construction is not dependent on an arbitrary choice of embedding while simultaneously often exhibiting less smeariness compared to the Fr\'echet mean. For the circle and the sphere, we show in Section \ref{sec:diff-means} that the diffusion mean for $t \to \infty$ converges to the extrinsic mean in the standard isometric embeddings. Thus, in these spaces, the diffusion means give an intrinsic interpretation of extrinsic means.

While the Fr\'echet mean is the center in the sense of minimizing squared distances, the diffusion mean, per its definition, should be interpreted as the most likely explanation of the data under the statistical model of Brownian motion. Since Brownian motions arise in the limit of random walks assuming only i.i.d. steps, this is arguably a very natural model of the dispersion between mean and observations with very weak assumptions on the probabilistic structure. In a sense much weaker than the Fr\'echet mean which measures dispersion between mean and observations by the energy of smooth geodesics. Both concepts generalize the Euclidean expected value to more general spaces, conserving in one case a metric property (Fr\'echet mean) and in the other case a probabilistic property (diffusion mean).
Using a likelihood criterion for defining means on manifolds has previously been explored in \citet{said_extrinsic_2012} on compact Lie groups, in \citet{sommer_bridge_2017} on landmark manifolds and general Riemannian manifolds \citet{sommer_modelling_2016} \citet{sommer_anisotropic_nodate}, and also in probabilistic principal component analysis \citet{zhang_probabilistic_2013}\citet{NyeTangWeyenbergYoshida2017}. The contributions of this paper are 
\begin{enumerate}
	\item A theoretical account of diffusion means on geodesically and stochastically complete connected Riemannian manifolds with bounded sectional curvature, and we explain how diffusion means can be defined on Lie groups lacking a bi-invariant Riemannian metric. 
	\item An exploration of the diffusion mean estimator, in particular the asymptotic behavior in terms of consistency and convergence rate, and an investigation of the effect of joint estimation of mean and diffusion variance on the convergence rate, where the latter essentially rules out smeariness.
	\item Proving that diffusion means admit positive mass on the cut locus.
	\item Showing that for zero diffusion, diffusion means reduce to Fr\'echet means.
	\item Several examples of improved convergence rate for bimodal distributions on $\cS^n$.
\end{enumerate}
In Section 2 we cover some basic background knowledge on differential geometry and geometric statistics to set the scene for the upcoming sections. A formal account of the diffusion means can be found in Section 3, covering both geodesically and stochastically complete connected Riemannian manifolds and other possible extensions, 
which are motivated by the possible application to computational anatomy and network theory, respectively. Furthermore, we prove that the diffusion mean admits positive mass on the cut locus contrary to the Fr\'echet mean.
 
Diffusion means, as MLEs, directly allow for the development of Bayes methods. Empirical Bayes methods based on shrinkage for Fr\'echet means in nonpositive curvature contexts studied by \citet{mccormack2021equivariant,yang2020empirical} could thus be extended.

An estimator of the diffusion mean is introduced in Section 4. We consider the estimator's asymptotic behavior and provide sufficient conditions for it to be strongly consistent in the sense of Ziezold (ZC) and Bhattacharya and Patrangenaru (BPC). We also provide sufficient conditions for the generalized CLT to hold for the diffusion estimator, thus giving conditions for $k$-smeariness of the estimator. 
We consider various bimodal distributions on $\cS^m$ and calculate the diffusion $t$-means and estimate the asymptotic convergence rate for various $t>0$. Furthermore, we 
give precise asymptotic results for the convergence of diffusion means to Fr\'echet means as $t\to 0$. The results and examples of the paper point to the importance of $t$ and we conclude the section by discussing the estimation of $t$ and the joint estimation of $\mu$ and $t$ parameters. This includes two examples of estimating $t$, one indicating that we are able to estimate the true variance of a \BND. We show that joint estimation of mean and diffusion variance rules out smeariness in all directions in general situations. We illustrate this by reducing finite sample smeariness in geomagnetism data. Finally, we show for the geomagnetism data as well as for wind direction data on the circle that, in consequence, the power of hypothesis tests increases.

\section{Background}
\subsection{Differential Geometry}
In this section we give a brief introduction to some of the essential notation and definitions in differential geometry. 
A Riemannian $m$-manifold $(\cM,g)$ is a smooth $m$-dimensional manifold $\cM$ equipped with a Riemannian metric $g$, i.e. a collection of inner products $g_y(u,v)= \langle u,v\rangle_y$ on each tangent space $T_y\cM$, such that $y\mapsto \langle u(y),v(y)\rangle_y$ is smooth for all sections $u,v$ of $T\cM$. 
Unless the dimension $m$ and metric $g$ are important, we omit these to simplify notation. The inner product induces norms on each tangent space, thus allowing for the definition of curve length $L(\gamma)= \int_0^1 \big|\big|\frac{d}{dt}\gamma(t) \big|\big|_{\gamma(t)}dt$ of a curve $\gamma:[0,1]\to \cM$. A curve is called a \emph{geodesic} if it is locally length minimizing everywhere. By assuming that $\cM$ is complete, there exists at least one length minimizing geodesic between every pair of points $x,y\in \cM$ and the Riemannian metric induces a distance $\dist(x,y)= \min_\gamma L(\gamma)$ where the minimum is taken over all geodesics joining $x$ and $y$. 
 
For each tangent space we define the \emph{exponential map} $\exp_y:T_y\cM\to \cM$ as the function mapping a tangent vector $v\in T_y\cM$ to the point reached in unit time by moving along the unique geodesic starting in $y$ with tangent vector $v$. The exponential map is a diffeomorphism in a star-shaped neighborhood of zero and we denote the maximal domain on which it is diffeomorphic as $\cD(y)\subset T_y\cM$. The image $\exp_y(\cD(y))\subset \cM$ covers all of $\cM$ except for the \emph{cut locus} $\mathcal{C}(y)$, which are all points $x$ such that a minimal geodesic from $y$ to $x$ is not longer minimal beyond $x$. Since the exponential map is a diffeomorphism, it has an inverse $\log_y:\cM\backslash \mathcal{C}(y)\to T_p\cM$ which again is a diffeomorphism, and it can therefore define a local 
coordinate system around $y$ whenever $T_y\cM$ is equipped with a 
basis.

The notion of Brownian motion extends to Riemannian manifolds, where it can be characterized in numerous ways. One way is through the minimal heat kernel, as it is the transition density of a Brownian motion. The heat kernel is the fundamental solution to the heat equation $\frac{d}{dt}u(x,y,t) = \frac{1}{2}\Delta_x u(x,y,t)$ where $\Delta$ is the Laplace Beltrami operator on $\cM$. The heat equation is not guaranteed to have a unique solution or any solutions at all. In order to ensure existence, we only consider stochastically complete manifolds as defined in the following definition. See \citet{hsu_stochastic_2002} for more detail. 
\begin{definition}\label{heatdef}
A Riemannian manifold $\cM$ is said to be \emph{stochastically complete} if there exists a minimal solution $p$ of the heat equation $\frac{d}{dt}u(x,y,t) = \frac{1}{2}\Delta_x u(x,y,t)$ satisfying $\int_\cM p(x,y,t)dy = 1$ for all $x\in\cM$ and $t>0$. We say that $p$ is the \emph{minimal heat kernel} of $\cM$. 
\end{definition}
Some of the more important properties of the minimal heat kernel, that we will be using throughout the paper, include:
\begin{enumerate}
	\item $p$ is strictly positive and $p\in \mathcal{C}^\infty(\cM\times \cM \times \R_+)$.
	\item Symmetry of the first two variables: $p(x,y,t)=p(y,x,t)$.
	\item $p$ is the transition density function of Brownian motion on $\cM$.
\end{enumerate}
Being the transition density of a Brownian motion, the heat kernel serves as a density of a generalized normal distribution on which we base our maximum likelihood method in the following section.

\subsection{Geometric Statistics}\label{GeomStat}
The field of (Euclidean) statistics widely relies on the assumption that the data can be embedded into a vector space. The assumption is both crucial to definitions such as the mean value, but also to the analysis of asymptotic behavior of estimators. Geometric statistics aims at generalizing the successful tools and results of Euclidean statistics to nonlinear spaces such as manifolds. Due to the different structure of nonlinear spaces, it is not always easy to determine unique or canonical generalizations of concepts from Euclidean statistics, and even basic concepts such as mean value and variance have lead to different generalizations.

One of the most common generalization of the mean to manifolds is the \emph{Fr\'echet mean (set)} $M = \argmin_{y\in \cM}\E[d(X,y)^2]$. Results concerning the existence and uniqueness \citet{karcher_riemannian_1977}\citet{kendall_probability_1990} pioneered the field by connecting the effect of curvature to the two properties. The Fr\'echet mean is estimated by the estimator $M_n =\argmin_{y\in \cM}\sum_id(X_i,y)^2$. Under certain assumptions, this is a strongly consistent estimator in the sense of both Ziezold (ZC) and Bhattacharya and Patrangenaru (BPC) \citet{ziezold_expected_1977} \citet{bhattacharya_large_2003}. An estimator $M_n$ of a set $M$ is said to be a strongly consistent estimator in the sense of Ziezold (ZC) if for almost all $\omega$
\begin{equation*}
\cap_{n=1}^\infty \overline{\cup_{k=n}^\infty M_{k}(\omega)} \subset M.
\end{equation*}
It is said to be a strongly consistent estimator in the sense of Bhattacharya and Patrangenaru (BPC) if $M\neq \emptyset$ and for every $\eps>0$ and almost all $\omega$ there is a number $n=n(\eps,\omega)>0$ such that 
\begin{equation*}
\cup_{k=n}^\infty M_k(\omega)\subset B(M,\eps)
\end{equation*}
where $B(M,\eps)$ is the union of balls of radius $\eps$ around all points in $M$. In Section~\ref{sec:asymptotics} we show that the diffusion mean estimator is strongly consistent in both senses under the same conditions as the Fr\'echet mean estimator. We also show that under corresponding assumptions, diffusion means converge for $t\to 0$ to the Fr\'echet means in the sense of Ziezold and Bhattacharya-Patrangenaru. The former is also called \emph{outer limit}, the latter \emph{one-sided Hausdorff limit}, cf. \citet{schotz_strong_2020,evans_strong_2020}.
 
The first generalization of the CLT to general manifolds was achieved by Bhattacharya and Patrangenaru (BP-CLT) \citet{bhattacharya_large_2005}. The BP-CLT establishes sufficient conditions for the Fr\'echet estimator to share the asymptotic behavior of the Euclidean average. However, the Fr\'echet mean estimator does not always exhibit this asymptotic behavior. It was first discovered for certain distributions on the circle \citet{hotz_intrinsic_2015}, that the estimator converged at a rate $n^{-\alpha}$ for $\alpha <1/2$ to a non-normal distribution. In order to theoretically account for this behavior, the definition of smeariness for estimators on manifolds and a generalized CLT was formulated by \citet{eltzner_smeary_2018}. The theorem does not only generalize to means of general loss functions $\E[\rho(y,X)]$, but also allows for general convergence rates $n^{-\alpha}$ and different limiting distribution.

A sequence of random variables $y_n\overset{\bP}{\to} y$ on a Riemannian manifold is said to be $k$-smeary if the sequence of random real valued vectors $n^{\frac{1}{2(k+1)}}\log_{y}(y_n)$ converges to a non-deterministic random vector in $\R^m$, see \citet{eltzner_smeary_2018} for more detail. When $k=0$, the convergence rate coincides with the usual rate $\sqrt{n}$ of the Euclidean average. 

The cornerstone in statistics underlying the CLT is the Gaussian normal distribution, due to its frequent occurrence in real data, but also due to its appearance in mean value estimation. One generalization of the normal distribution is the Normal Law \citet{pennec_intrinsic_2006}, which utilizes the exponential map to map a normal distribution on $\R^m$ to the manifold. The \BND\ uses the natural connection between the Brownian motion and normal distribution in Euclidean spaces. The Brownian normal probability distribution coincides with the end point distribution of a Brownian motion, thus we can generate samples from the generalized normal distribution by sampling from a Brownian motion.

\section{Diffusion means}\label{sec:diff-means}
Throughout this paper $\cM$ denotes a geodesically and stochastically complete and connected Riemannian manifold with bounded sectional curvature. This ensures existence of a minimal heat kernel $p$ that is bounded.
  
In this section we define the diffusion mean of random variables on $\cM$ and we describe how the diffusion means connect to both the Fr\'echet mean and the generalized Fr\'echet means. We provide a registry of manifolds where a closed form for the heat kernel have been obtained. This includes the hyperspheres and the hyperbolic spaces. In Subsection~\ref{subsec:diff_means_on_spheres} we focus on distributions on the sphere and calculate the diffusion means for the two pole distribution in Figure \ref{2poledis}. We conclude the section by mentioning other types of spaces for which the diffusion means can be defined in Subsection~\ref{subsec:diff_means_on_lie_and_graphs}.

The definition of diffusion means is inspired by the Euclidean maximum likelihood estimation and the mean is one of the parameters of a normal density which can be estimated from data. In the general case of a manifold data space, we consider the generalized normal distribution and therefore define diffusion means as the most likely origin points of a Brownian motion hitting the observed data. 
 
For the random variable $X$ on $\cM$, with underlying probability space $(\Omega, \cF, \bP)$, we define the \emph{log-likelihood function} $L_t:\cM \to \R$ to be 
\begin{equation}\label{negloglikelihood}
L_t(y) = \E[-\ln(p(X,y,t))]
\end{equation}
for each $t>0$ where $p$ denotes the minimal heat kernel from Definition \ref{heatdef}. This gives rise to the diffusion means as the global minima of the log-likelihood function, i.e. the points maximizing the log-likelihood of a Brownian motion. Notice that we negate the logarithmic heat kernel in the log-likelihood function, in order to characterize diffusion means as minima instead of maxima. 
\begin{definition}
	With the probability space $(\Omega, \cF,\bP)$, let $X$ be a random variable on $\cM$ and fix $t>0$. The \emph{diffusion $t$-mean set} $E_t(X)$ of $X$ is the set of global minima of the log-likelihood function $L_t$, i.e. 
	\begin{equation*}
	E_t(X) = \argmin_{y\in \cM}\E[-\ln(p(X,y,t))].
	\end{equation*}
	If $E_t(X)$ contains a single point $\mu_t$, we say that $\mu_t$ is the \emph{diffusion $t$-mean} of $X$. 
\end{definition}
Notice, that for each $t>0$ there is an associated mean set $E_t$, leading to a parametric family of mean sets. In contrast, the Fr\'echet mean does not have an additional parameter and one gets only one mean set. The diffusion means provide a natural extension of the Fr\'echet mean, due to the convergence of the diffusion means set to the set of classical Fr\'echet means in the sense of Ziezold and in the sense of Bhattacharya and Patrangenaru, for which we present sufficient conditions later in Subsection \ref{shorttime}. 

The time parameter $t$ can be interpreted as a \emph{variance} parameter. This is because a Brownian motion is more likely to travel far on the manifold for large $t$, whereas it is more likely to be centralized around its origin for small $t$. This means that for small $t$, the diffusion $t$-mean is sensitive to probability mass far from the mean, which is also the case for the Fr\'echet mean. For larger $t$ however, this sensitivity is diminished, which we exemplify later in Example \ref{ex:two_points1}. 
\begin{remark}\label{RmkOnBounded}
	In the language of generalized Fr\'echet means from \citet{huckemann_intrinsic_2011}, the definitions above correspond to choosing $\rho(x,y)^2 = - \ln(p(x,y,t)\maxt^{-1})$ for a fixed $t>0$ and constant $\maxt >0$. For this to make sense however, we must ensure existence of a $\maxt>0$ such that $\ln(p(x,y,t)\maxt^{-1})\leq 0$. If for all $t>0$, the heat kernel $p(x,y,t)$ is bounded, then letting $\maxt$ be a bound on $p(x,y,t)$ for each $t>0$ solves this issue since
	\begin{align*}
	\argmin_{y\in \cM} \E [- \ln(p(X,y,t)\maxt^{-1})] = \argmin_{y\in \cM} \E [- \ln(p(X,y,t))]. 
	\end{align*}
	The heat kernel is bounded on compact manifolds, but not necessarily on non-compact complete manifolds. It follows from \citet[Theorem 4]{cheng_upper_1981}, that bounded sectional curvature suffices for $\maxt>0$ to exist for all $t>0$, which is why we include this restriction on $\cM$. This condition is not limiting, since the existence of such a constant is also necessary for a diffusion $t$-mean to exist. 
\end{remark}
Instead of fixing $t>0$, we can also minimize the log-likelihood function $(y,t)\mapsto L_t(y)$ as a function of both $t$ and $y$, thus obtaining the optimal diffusion parameters 
\begin{equation}
  D(X) = \argmin_{(y,t)\in \cM\times \R_+} \mathbb E[-\ln p(X,y,t].
\end{equation}
This paper mostly regards the properties and asymptotic behavior of estimating the diffusion mean by itself, but we do include an example of estimating the parameters simultaneously in Section \ref{scn:application}. 

The closed form of the heat kernel has not been obtained for general manifolds. In the following five examples, we take a look at some of the spaces for which a closed form has been obtained. 
\begin{example}
	The heat kernel $p$ on the Euclidean space $\R^m$ is given by the function
	\begin{equation*}
	p(x,y,t) = \left(\frac{1}{(2\pi t)^{m/2}}\right) e^{\frac{-|x-y|^2}{2t}}. 
	\end{equation*}
	for $x,y\in \R^m$ and $t>0$. The diffusion $t$-means of a random variable $X$ do not depend on $t$ and coincide with the expected value $\E[X]$ since 
	\begin{align*}
	\argmin_{y\in \R^m} L_t(y)&=\argmin_{y\in \R^m} \int_{\R^m} -\frac{2}{m}\ln(2\pi t) + \left(\frac{|x-y|^2}{2t}\right)d\bP_X(x) \\
	&= \argmin_{y\in \R^m} \int_{\R^m} |x-y|^2d\bP_X(x)
	= \argmin_{y\in \R^m} \E[(X-y)^2]
	\end{align*}
	Thus, $\mu_t = \E[X]$ for all $t>0$. 
\end{example}
\begin{example}
	The heat kernel on the circle $\cS^1$ is given by the wrapped Gaussian
	\begin{equation}\label{wrap-Gauss:eqn}
	p(x,y,t) = \frac{1}{\sqrt{2\pi t}} \Big( \sum_{k\in \Z}\exp\Big(\frac{-(x-y + 2\pi k)^2}{2t} \Big) \Big)
	\end{equation}
	for $x,y\in \R /2 \pi \Z\cong \cS^1$ and $t>0$, and the log-likelihood function for the random variable $X:\Omega \to \cS^1$ becomes
	\begin{align*}
	L_t(y) &= - \int_{\cS^1} \ln \Big(\frac{1}{\sqrt{2\pi t}} \Big( \sum_{k\in \Z}\exp\Big(\frac{-(x-y + 2\pi k)^2}{2t} \Big) \Big) \Big) d\bP_X(x) \\
	&= \ln(\sqrt{2\pi t})+ \int_{\cS^1} \left[ \frac{(x-y)^2}{2t} - \ln \left(\sum_{k\in \Z}\exp\Big(\frac{-(2\pi k)^2}{2t}\Big)\exp\Big(\frac{-4\pi k (x-y)}{2t} \Big) \right) \right] d\bP_X(x).
	\end{align*}
	Notably, even on this simple space, the $t$-dependence cannot be split off into a separate summand and $\mu_t$ is therefore explicitly dependent on $t$.
\end{example}
\begin{example}
	The heat kernel on the spheres $\cS^{m}$ for $m\geq 2$ can be expressed as the uniformly and absolutely convergent series \citet[Theorem 1]{zhao_exact_2018}
	\begin{equation}\label{sphereheat}
	p(x,y,t) = \sum_{l=0}^\infty e^{-l(l+m-1)\frac12t}\frac{2l+m-1}{m-1} \frac{1}{A_{\cS}^{m}} C_l^{(m-1)/2}(\langle x,y\rangle_ {\R^{m+1}} )
	\end{equation}
	for $x,y\in \cS^{m}$ and $t>0$, where $C_l^\alpha$ are the Gegenbauer polynomials and $A_{\cS}^{m} = \frac{2\pi^{(m+1)/2}}{\Gamma((m+1)/2)}$ the surface area of $\cS^{m}$. For $m=2$, the Gegenbauer polynomials $C^{1/2}_l$ coincide with the Legendre polynomials $P^0_l$ and the heat kernel on $\cS^2$ is 
	\begin{equation}
	p(x,y,t) = \sum_{l=0}^\infty e^{-l(l+1)\frac12t}\frac{2l+1}{4\pi} P^0_l(\langle x,y\rangle_ {\R^3} ).
	\end{equation}
	Again, $\mu_t$ is explicitly dependent on $t$ on these spaces.
\end{example}
\begin{example}
	The heat kernel on the hyperbolic space $\mathbb{H}^n$ can be expressed by the following formulas \citet{grigoryan_heat_1998} for $n\geq 1$. For odd $n =2m+1$, the heat kernel is given by
	\begin{equation*}
	p(x,y,t) = \frac{(-1)^m}{2^m\pi^m}\frac{1}{\sqrt{2\pi t}}\frac{\rho}{\sinh\rho}\left(\frac{1}{\sinh\rho} \frac{\partial}{\partial\rho} \right)^m e^{-m^2t-\frac{\rho^2}{2t}}
	\end{equation*}
	where $\rho = \dist_{\mathbb{H}^n}(x,y)$ and for even $n=2m+2$, it is given by 
	\begin{equation*}
	p(x,y,t) = \frac{(-1)^m}{2^{m+1}\pi^{m+\frac32}}t^{-\frac32}\frac{\rho}{\sinh\rho}\left(\frac{1}{\sinh\rho} \frac{\partial}{\partial\rho} \right)^m  \int_\rho^\infty \frac{se^{-\frac{s^2}{2t}}}{(\cosh s-\cosh \rho)^\frac12}ds.
	\end{equation*}
	Again, $\mu_t$ is explicitly dependent on $t$ on these spaces.
\end{example}
\begin{example} The fundamental solution to the heat equation on a Lie group $G$ of dimension $m$ is \citet{arede_heat_1991}
	\begin{equation*}
	p(x,e,t) = (2\pi t)^{-m/2} \prod_{\alpha\in\Sigma^+} \frac{i\alpha(H)}{2\sin(i\alpha(H)/2)}\exp\left(\frac{\|H\|^2}{2t}+\frac{\|\rho\|^2 t}{2}\right) \cdot E_x(X_\tau >t)
	\end{equation*}
	where $e$ is the neutral element, $x = \exp(Ad(g)H)\in G-\mathcal{C}(e)$ for some $g\in G$, $\Sigma^+$ is the set of positive roots, $\rho = \sum_{\alpha\in\Sigma^+} \alpha$, and $\tau$ is the hitting time of $\mathcal{C}(e)$ by $(x_s)_{0\leq s\leq t}$. Lie groups are relevant data spaces for many application, e.g. in the modeling of joint movement for robotics, prosthesis development and medicine.
\end{example}
The symmetry in $x$ and $y$, which exists by construction, is clearly visible in the examples above. Moreover, when $t>0$ and $x\in \cM$ where $\cM = \R^n, \cS^n$ or $\mathbb{H}^n$, the heat kernel only depends on the geodesic distance \citet{alonso-oran_pointwise_2019}. 

\begin{remark}
	For spaces where a closed form has not been obtained, we can turn to various estimates. We include some of the most important below. 
	\begin{enumerate}
		\item For complete Riemannian manifolds of dimension $m$, we have the asymptotic expansion \citet[Theorem 5.1.1]{hsu_stochastic_2002} 
		\begin{equation*}
		p(x,y,t)\sim \left(\frac{1}{2\pi t} \right)^{\frac{d}{2}}e^{\frac{-d(x,y)^2}{2t}}\sum_{n=0}^{\infty} H_n(x,y)t^n
		\end{equation*}
		on compact subset with $x\notin \mathcal{C}(y)$. Here $H_n$ are smooth functions satisfying a recursion formula \citet{chavel_eigenvalues_1984} with $H_0(x,y) = \sqrt{J(\exp_x)(\exp^{-1}_x(y))}$ and $J$ denoting the Jacobian.
		\item Assuming also non-negative Ricci curvature, the heat kernel is bounded \citet[Corollary 3.1]{li_parabolic_1986}
		\begin{equation*}
		p(x,y,t) \leq \frac{C(\eps)}{\text{Vol}(x,\sqrt{t})}\exp \left(-\frac{\dist(x,y)^2}{(2+\eps)t}\right)
		\end{equation*}
		for $0<\eps<1$ and constant $C(\eps)$, where $\text{Vol}(x,\sqrt{t})$ denotes the volume of the ball around $x$ of radius $\sqrt{t}$.
		\item Using bridge sampling, the heat kernel can be estimated by the expectation over a guided process \citet{delyon_simulation_2006}, \citet{sorensen_importance_2012}. An example of this is the estimated heat kernel on high-dimensional landmark manifolds \citet{sommer_bridge_2017}. 
		This approach can be carried even further to allow direct sampling of an approximation of the diffusion mean, thus avoiding nested optimization as is in general needed for the Fr\'echet mean \citet{jensenMeanEstimationDiagonal2022}.
	\end{enumerate}

  It is natural to compare the computational cost of estimating the Fr\'echet mean and the diffusion mean. The Fr\'echet mean is by no means simple to compute in spaces without closed form solutions of geodesic distances. In such spaces, the computation in general involves iterative minimization of the Fr\'echet variance which in turn requires evaluating the Riemannian logarithm at all data points. When these logarithms have no closed form, this results in a nested optimization problem which can be very expensive to solve numerically. The use of bridges where the heat kernel is approximated via Monte Carlo sampling of the Brownian motion conditioned on hitting the data points is quite efficient, even on high-dimensional manifolds. The computational expense is not directly comparable to computing the Riemannian logarithm (it depends on factors including the numerical bridge simulation scheme and number of simulations per bridge). However, as a practical example, the diffusion mean including the variance parameter $t$ can be found via bridge sampling on $\mathbb S^2$ (not using the heat kernel expansion) for 100 data points in roughly 15 s. on a laptop. In addition to this, the diffusion mean can be approximated by direct sampling, see \cite{jensenMeanEstimationDiagonal2022}. This is in general much faster than computing the Fr\'echet mean, at the expense of additional uncertainty in the mean estimates.\footnote{Example code for both iterative optimization using bridge sampling and direct sampling can be found at \url{https://bitbucket.org/stefansommer/jaxgeometry/src/main/diffusion_mean.ipynb}.}
\end{remark}

\subsection{Diffusion means on spheres}\label{subsec:diff_means_on_spheres}

We here take a closer look at the diffusion means on the hyperspheres $\cS^{m}$ for $m\geq 1$. 

\subsubsection{Relation to Fr\'echet mean and extrinsic mean}

We prove that the diffusion mean interpolates between the Fr\'echet mean and the extrinsic mean on spheres as the diffusion time parameter $t$ varies between $0$ and $\infty$. 
The result is illustrated for simulated data in Figure \ref{fig:int-ext}.
The auxiliary lemmas \ref{lem:wrapped-normal}, \ref{lem:rho1-derivative}, and \ref{lem:rhom-derivative} can be found in Appendix~\ref{subsec:thm32}.

\begin{figure}[ht]
	\centering
	\includegraphics[width=0.4\linewidth]{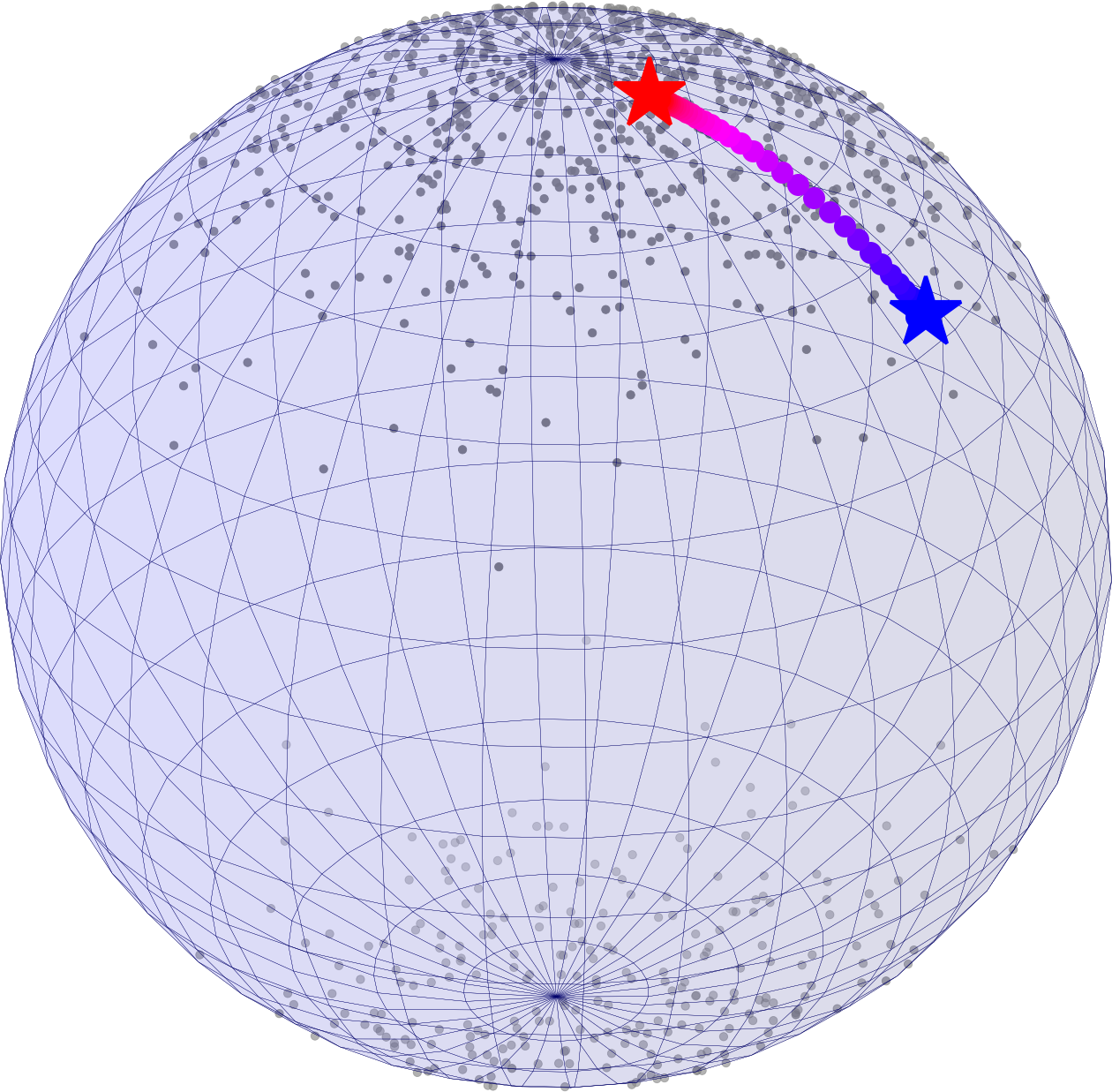}
	\caption{For the simulated sample, displayed in gray, the blue star indicates the intrinsic mean and the red star indicates the extrinsic mean. The colored points indicate diffusion means for different $t$, where small values of $t$ have a bluer hue and large values of $t$ have a redder hue.}
	\label{fig:int-ext}
\end{figure}
\begin{theorem} \label{thm:diff-ext}
  Consider a random variable on $S^m$ with $m \ge 1$. Then the diffusion mean set converges to the extrinsic mean set for $t \to \infty$.
\end{theorem}

\begin{proof}
    For $S^1$, we use the form for the wrapped normal density shown in Lemma \ref{lem:wrapped-normal} and set $s := e^{-t/2}$, so the limit $t \to \infty$ is equal to the limit $s \to 0$. Then, consider the function
    \begin{align}
      \rho(x,y,s) :=& - 1/s \log ( 2 \pi p(x,y,t(s)) ) \nonumber\\
      =& - 1/s \sum_{k=1}^\infty \left( \log\left(1-s^{2k}\right) + \log \left( 1 + 2 s^{2k-1} \cos(x-y) + s^{4k-2}\right) \right) \label{eq:rho-s1} \, .
    \end{align}
    In Lemma \ref{lem:rho1-derivative} we show that the derivative $\frac{d}{ds} \rho(x,y,s)$ exists and is uniformly bounded for any $x,y \in [-\pi, \pi)$ and $s \in [0,1/2]$.
    
    For $S^m$ with $m > 1$, let $s := e^{-\frac{m}{2}t}$ and equivalently $t(s) := - \frac{2}{m} \log(s)$. Then, consider the function
    \begin{align}
      \rho(x,y,s) :=& - 1/s \log \left( A_{\cS}^{m} p(x,y,t(s)) \right)\nonumber\\
      =& - 1/s \log \left(\sum_{k=0}^\infty e^{-k(k+m-1)t(s)/2}\frac{2k+m-1}{m-1} C_k^{(m-1)/2}(\langle x,y\rangle_ {\R^{m+1}} ) \right) \nonumber\\
      =& - 1/s \log \left( \sum_{k=0}^\infty s^{k}s^{\frac{k(k-1)}{m}}\frac{2k+m-1}{m-1} C_k^{(m-1)/2}(\langle x,y\rangle_ {\R^{m+1}} ) \right) \label{eq:rho-sm} \, .
    \end{align}
    In Lemma \ref{lem:rhom-derivative} we show that the derivative $\frac{d}{ds} \rho(x,y,s)$ exists and is uniformly bounded for any $x,y \in S^m$ and $s \in [0,1/2]$.
    
    In both cases, we have for any fixed $s$
    \begin{align*}
      \argmin_{z \in S^m} \limits L_{t(s)}(z) = \argmin_{z \in S^m} \limits \mathbb{E}[-\log(p(X,z,t(s))] = \argmin_{z \in S^m} \limits \mathbb{E}[\rho(X,z,s)] \, .
    \end{align*}
    Furthermore, from the existence and uniform boundedness of the derivative in $s$, we derive a uniform continuity property as follows
    \begin{align*}
      \forall \eps > 0 \, \exists \delta > 0 \, \forall x,y \in S^m \, \forall s \in [0,\delta] \, : \, |\rho(x,y,s) - \rho(x,y,0)| < \eps \, ,
    \end{align*}
    because $S^m$ is compact. Now, we can derive
    \begin{align*}
      \left|\mathbb{E}[\rho(X,z,s)] - \mathbb{E}[\rho(X,z,0)]\right| \le \mathbb{E}[|\rho(X,z,s) - \rho(X,z,0)|] < \mathbb{E}[\eps] = \eps
    \end{align*}
    and thus
    \begin{align*}
      y \in \argmin_{z \in S^m} \limits \mathbb{E}[\rho(X,z,\delta)] \quad \Leftrightarrow \quad &\mathbb{E}[\rho(X,y,0)] < \mathbb{E}[\rho(X,y,\delta)] + \eps = \inf_{z \in S^m} \limits \mathbb{E}[\rho(X,z,\delta)] + \eps\\
      <& \inf_{z \in S^m} \limits \mathbb{E}[\rho(X,z,0)] + 2\eps \, .
    \end{align*}
    It is thus clear that in the limit $\eps \to 0$, the set $\argmin_{z \in S^m} \limits \mathbb{E}[\rho(X,z,\delta)]$ must converge to the set $\argmin_{z \in S^m} \limits \mathbb{E}[\rho(X,z,0)]$ in the Hausdorff topology.
    
    This ensures that
    \begin{align*}
      \lim_{s \to 0} \limits \argmin_{z \in S^m} \limits \mathbb{E}[\rho(X,z,s)] =\argmin_{y \in S^m} \limits \lim_{s \to 0} \limits \mathbb{E}[\rho(X,z,s)] \, .
    \end{align*}
    
    For the circle, we recall Equation \eqref{eq:rho-s1} and use the expansion $\log(1+x) = x + \mathcal{O}(x^2)$
    \begin{align*}
      \rho(x,y,s) =& - \frac{1}{s} \sum_{k=1}^\infty \left( -s^{2k} + \mathcal{O}(s^{4k}) + 2 s^{2k-1} \cos(x-y) + \mathcal{O}(s^{4k-2}) \right) = - 2 \cos(x-y) + \mathcal{O}(s) \, .
    \end{align*}
    Thus $\rho(x,y,0) = - 2 \cos(x-y)$. Embedding $S^1$ in $\mathbb R^2$ via setting $W=(\cos X,\sin X), u = (\cos y,\sin y)$, we have that
    \begin{equation*}
      \argmin_{z\in S^1} \mathbb E[\rho(X,z,0)] =\mathop{\rm argmax}_{u\in \mathbb R^2, \|u\|=1} \mathbb E[\langle W,u\rangle]
    \end{equation*}
    Constrained maximization yields the maximizer $\frac{\mathbb E[W]}{\|\mathbb E[W]\|}$ in case of $\mathbb E[W]\neq 0$ and all of $S^1$ else. In both cases this is the extrinsic mean, as asserted.
    
    For the sphere, to see that the diffusion mean approaches the extrinsic mean as $t \to \infty$, we highlight the lowest orders in $s$ in the series expansion \eqref{eq:rho-sm} and use $\log(1+x) = x + \mathcal{O}(x^2)$ as well as the Gegenbauer polynomial $C_1^{(m-1)/2} = (m-1)x$
    \begin{align*}
      \rho(x,y,s) =& - \frac{1}{s} \log \left( 1 + s (m+1) \langle x,y\rangle_{\R^{m+1}} + \mathcal{O}(s^2) \right) = - (m+1) \langle x,y\rangle_{\R^{m+1}} + \mathcal{O}(s)
    \end{align*}
    so that
    \begin{equation*}
      \argmin_{z\in S^m} \mathbb E[\rho(X,z,0)] =\mathop{\rm argmax}_{z\in \mathbb R^{m+1}, \|z\|=1} \mathbb E[\langle X,z\rangle]\,.
    \end{equation*}
    Constrained maximization yields again the maximizer $\frac{\mathbb E[X]}{\|\mathbb E[X]\|}$ in case of $\mathbb E[X]\neq 0$ and all of $S^m$ else. In both cases this is the extrinsic mean, as asserted.
\end{proof}

Diffusion means describe centers of diffusion processes, which suggests interpreting the intrinsic and extrinsic means as centers of diffusion in the limits of vanishing and diverging diffusion variance $t$, respectively. If one estimates the diffusion means simultaneously with the variance, this will yield a diffusion mean close to the intrinsic mean for very concentrated data and close to the extrinsic mean for very spread out data. Based on this, if the choice is only between intrinsic and extrinsic mean, one may heuristically suggest to prefer intrinsic means as location statistics for concentrated data and extrinsic means for very spread out data.

In \cite{H_meansmeans_12} it was found that for rather concentrated data, the \emph{residual mean} is close to a geodesic passing through the extrinsic and intrinsic mean but on the opposite side of the extrinsic mean from the intrinsic mean. Our findings raise the question for future research, whether diffusion means also connect to residual means.

\subsubsection{Two pole distribution}
We now turn to studying diffusion means of the two pole distribution in Figure \ref{2poledis}, particularly because the behavior of the Fr\'echet mean has been studied for this distribution. 
We begin by introducing some useful notation and the choice of coordinates that will be used throughout the paper. 

We choose non-standard polar coordinates $\theta_1,...,\theta_{m-1}\in [-\pi/2,\pi/2]$ and $\phi \in [-\pi,\pi)$ of $\cS^{m}$
\begin{equation}\label{coordinates}
x = \begin{pmatrix}
- \Big(\prod_{j=1}^{n-2}\cos\theta_j\Big) \cos\phi \\
- \Big(\prod_{j=1}^{n-2}\cos\theta_j\Big) \sin\phi \\
\vdots \\
-\cos\theta_1\cos\theta_2\sin\theta_3 \\
-\cos\theta_1\sin\theta_2 \\
\sin\theta_1
\end{pmatrix}
\end{equation}
such that the north pole $\mu=(0,1,0,...,0)$ has polar coordinates $(0,...,0,-\pi/2)$. Using these coordinates we may, due to rotation symmetry around the north pole, assume w.l.o.g. that any fixed point $y\in \cS^{m}$ has coordinates $(0,...,0,-\pi/2+\delta)$ for some $\delta \in [0,\pi]$. We denote these coordinates by $y_\delta$. 
\begin{remark}\label{heatdecreasing}
	It follows from \citet[Theorem 1.3]{alonso-oran_pointwise_2019}, that the map $\delta\in [0,\pi] \mapsto p(\mu, y_\delta,t)$ is decreasing for each fixed $t>0$, since the geodesic distance is $\dist(\mu, y_\delta)= \arccos\langle \mu,y_\delta\rangle = \delta$ for $\delta\in [0,\pi]$. 
\end{remark}
To simplify the notation of the heat kernel on $\cS^m$ as written in Equation \eqref{sphereheat} and the logarithmic heat kernel, we define the functions $\hfun, \lfun:[-1,1]\to \R$ for each $t>0$ and integer $m\geq 2$ by 
\begin{equation*}
\hfun(x) = \sum_{l=0}^\infty e^{-l(l+m-1)\frac12t}\frac{2l+m-1}{m-1} \frac{1}{A_{\cS}^{m}} C_l^{(m-1)/2}(x), \quad \lfun(x) = \ln \hfun(x).
\end{equation*}
Then $p(x,y,t) = \hfun(\langle x,y \rangle_{\R^{m+1}})$ and the log-likelihood function $L_t$ of a random variable $X$ on $\cS^m$ with density $p_X(x)$ can written in terms of $\lfun$ as 
\begin{equation*}
L_t(y) = \int_{\cS^m} -\lfun(\langle x,y \rangle_{\R^{m+1}}) p_X(x)dx.
\end{equation*}
The following two lemmas uncover some useful properties of $\hfun$ and $\lfun$. The lemmas are proved in Appendix~\ref{appendix:sec3proofs}.
\begin{lemma}\label{loghdiff}
	The derivative $\frac{d}{dx}\lfun(x)$ is positive for all $m\geq 2$ and $t>0$.
	For all integers $m\geq2$, define 
	\begin{equation}\label{boundont}
	\delta(m)= \begin{cases}
	\frac{2}{(m+1)} \ln \left((m+1)+4(m+3)\right), &\text{ if } m=2,3 \\ \ln\left(\frac{32(m+3)}{9(m+1)}\right), &\text{ if }m\geq 4
	\end{cases}
	\end{equation}
	The second derivative $\frac{d^2}{dx^2}\lfun(x)$ is negative for all $m\geq 2$ and $t>\delta(m)$. 
\end{lemma}
\begin{lemma}\label{positivedifferentials}
	For $m\geq 2$ and $n=0,1$, the map $\frac{d^n}{dx^n}\hfun(x)$ is positive for all $t>0$. For $n\geq 2$ the map is positive for $t\geq 1$. 
\end{lemma}

In the following example we will calculate the diffusion means of the distribution in Figure \ref{2poledis}, which consists of mass in two opposite poles on $\cS^m$. The Fr\'echet mean is unique only if the south pole mass $\alpha$ is zero. When the mass is non-zero, the Fr\'echet mean set is a ring around the pole with radius depending on $\alpha$. In fact, a point $\mu$ on a compact manifold $\cM$ cannot be a Fr\'echet mean of a distribution with positive mass on $\mathcal{C}(\mu)$ \citet{le_measure_2014}. The following example proves that this is not the case for diffusion means, and it shows the influence of the time parameter $t$ on the diffusion mean set.
\begin{example}\label{ex:two_points1} Fix $m\geq 2$ and $t>\delta(m)$, and
	consider the random variable $X$ on $\cS^{m}$ distributed according to the two pole distribution $P(X\in -\mu) = \alpha$ and $P(X\in \mu) = 1-\alpha$ displayed in Figure \ref{2poledis}. The log-likelihood function of $X$ at $y_\delta$ can be expressed as a map of $\delta\in [0,\pi]$,
	\begin{align*}
	L_{t,m}(y_\delta) &= -(1-\alpha) \lfun(\cos \delta)-\alpha \lfun(-\cos \delta) =: \widehat{L}_{t,m}(\delta) 
	\end{align*}
	since $\langle y_\delta, \mu\rangle_{\R^m} = \cos\delta$. Figure \ref{2poleLikelihood} provides a plot of the graph for $t=3$ and various values of $\alpha$. 
	The first derivative with respect to $\delta$ is
	\begin{align}\label{Gder}
	\widehat{L}'_{t,m}(\delta) &= \sin\delta\left( (1-\alpha)\lfun'(\cos\delta) -\alpha \lfun'(-\cos\delta) \right)
	\end{align}
	Clearly $\widehat{L}'_{t,m}(0)=\widetilde{L}_{t,m}'\textbf{}(\pi) = 0$. 
	For $t>\delta(m)$ as defined in Equation \eqref{boundont}, it follows by Lemma \ref{loghdiff} that $\widehat{L}_{t,m}'(\delta)$ is strictly positive on $(0,\pi)$ whenever 
	\begin{equation*}
	\alpha \leq \Lambda_m(t):=\frac{\hfun(-1)\hfun'(1)}{\hfun'(-1)\hfun(1)+\hfun(-1)\hfun'(1)}
	\end{equation*}
	Thus, there is a unique minimum of $L_{t,m}$ at $\mu$ for $\alpha \leq \Lambda_m(t)$, which means that $\mu$ is the unique diffusion $t$-mean when the weight is sufficiently small. The functions $\Lambda_m$ can be seen in Figure \ref{2polesmeary} for $m=2,3,4,5$. Furthermore, we have $\lim_{t \to \infty}\Lambda_m(t) = \frac{1}{2}$ for all $m\geq 2$, which implies that for any $ \alpha < \frac12$ there exists a $t_\alpha>0$ such that $\mu$ is the unique diffusion $t_\alpha$-mean of the distribution with this particular $\alpha$. This means that we can pick the parameter $t$ to allow for almost arbitrary mass on the opposite pole of the mean. Finally, let us consider what happens for $\alpha = \frac12$, where $\widehat{L}_{t,m}$ simplifies to 
	\begin{equation*}
	\widehat{L}_{t,m}'(\theta) = \frac{1}{2}\sin\theta \left(\lfun'(\cos\theta)- \lfun'(-\cos\theta) \right).
	\end{equation*}
	Clearly $\widehat{L}_{t,m}'(0)=\widehat{L}_{t,m}'(\pi/2)=\widehat{L}_{t,m}'(\pi)=0$. Still assuming $t>\delta(m)$, it follows directly from Lemma \ref{loghdiff} that $\lfun'(\cos\theta)- \lfun'(-\cos\theta)$ is negative for $\theta\in(0,\pi/2)$ and positive for $\theta\in(\pi/2,0)$. 
	This implies that $\delta=\pi/2$ is the only minimum of $\widetilde{L}_{t,m}$ and that the diffusion $t$-mean set for $t>\delta(m)$ is the set of all of the points on the equator between $\mu$ and $-\mu$. 
	\begin{figure}[ht]
		\centering
		\includegraphics[height = 5cm,clip]{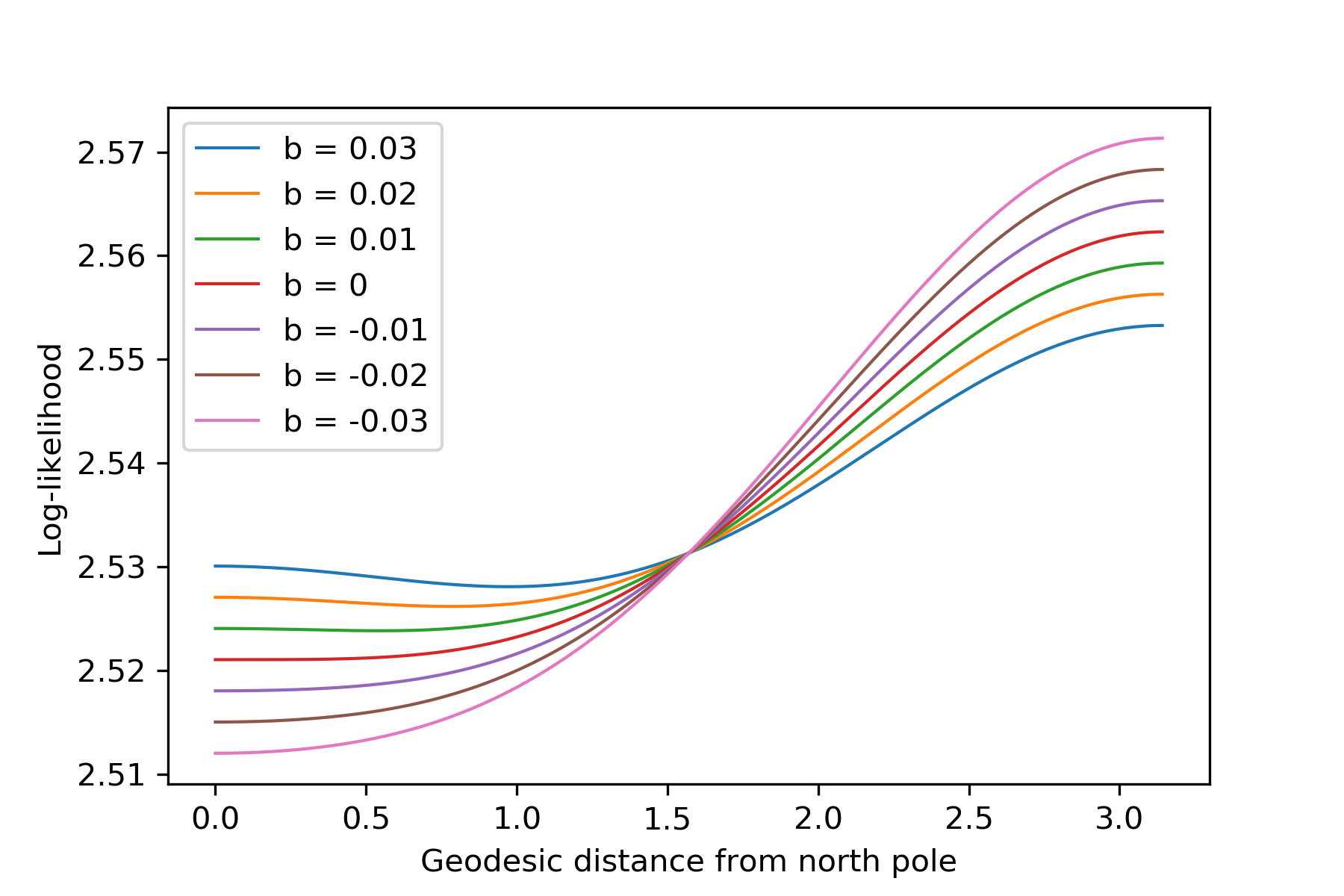}
		\caption{The log-likelihood function for the two pole distribution on $\cS^2$ with $t=3$ and weight $\alpha = \Lambda_2(t)+b$ on the south pole for various $b$.}
		\label{2poleLikelihood}
	\end{figure}
\end{example}
\subsection {Diffusion means on Lie groups and graphs}\label{subsec:diff_means_on_lie_and_graphs}
The assumption that the data space is a well-behaved smooth manifold is convenient for the existence of a heat kernel, however the assumption is not necessary for the construction. Diffusion means do not dependent on the smooth structure of a manifold or even the existence of a metric. The definition can be extended to random variables on spaces where it is possible to take identically distributed and independent steps. In this section we take a look at how the diffusion $t$-means can be defined for random variables on more general spaces.
 
\begin{enumerate}
	\item Defining mean sets on non-Riemannian affine connection manifolds is of great interest, due to the direct application in fields such as computational anatomy. There exists a symmetric bi-invariant connection on all Lie groups, which allows for an affine connection structure of the Lie group which may not be metric, meaning that the Riemannian barycenters are not defined. Due to the affine connection structure, it is still possible to define a mean value, for example through exponential barycenters \citet{pennec_barycentric_2015,pennec_barycentric_2018}. As diffusion means only depend on the likelihood resulting from a Brownian motion, this structure also allows for random walks to exist, therefore also for the possibility of defining diffusion means. A specific example of such a non-compact Lie group is the special Euclidean group $SE(n)$ (rigid transformations in $\mathbb R^n$) of which there exists no bi-invariant Riemannian metric \citet{pennec_exponential_2013}. In contrast, the bi-invariant connection defines parallel transport and horizontal vector fields on the frame bundle $FSE(n)$ of the group. A Brownian motion can be defined by stochastic development using the horizontal fields, and a choice of measure on the group implies a smooth solution to the corresponding heat equation on a sub-bundle of $FSE(n)$. This in turn defines a version of heat flow on $SE(n)$ and thereby a diffusion mean.
	\item Motivated by its possible application in network theory, the definition of diffusion means can also be extend to graphs. Let $G=(V,E)$ be a graph of $n = |V|$ vertices and $m = |E|$ edges and let $d_i$ denote the number of edges connected to vertex $v_i$ and $d_{ij}$ denote the number of edges between $v_i$ and $v_j$. Define $P$ to be the matrix with entries $P_{ij} = \frac{d_{ij}}{d_{i}}$. This is exactly the transition density of a random walk on $G$ for $t=1$. Allowing for an arbitrary number of steps $t\in \N$, the transition density is $P^t$. Thus, starting in $v_i$, the probability of hitting $v_j$ in time $t$ with a random walk is $(P^t)_{ij}$. 
	
	The distribution of a stochastic variable $X$ on the graph $G$ is given by the probabilities $\Pr(X = v_j)$ for $j=1,...,n$. To define diffusion $t$-mean sets, we want to find the $i$ maximizing the likelihood function
    $
	L_t(i) = \sum_{j=1}^{n} \Pr(X=v_j)(P^t)_{ij} = e_i P^t p^X
    $
	where $p^X\in \R^n$ has entries $p^X_j = \Pr(X = v_j)$. Thus for each $t\in \N$, we define the diffusion $t$-means of $X$ to be $E_{t} = {\argmin}_{1\leq i\leq n} e_iP^tp^X.$ 
\end{enumerate}

\section{Asymptotic analysis}\label{sec:asymptotics}

In this section we define the \emph{diffusion estimator} and investigate its asymptotic behavior. For manifolds of bounded sectional curvature, we present sufficient conditions for strong consistency in the sense of Ziezold and in the sense of Bhattacharya and Patrangenaru as defined in Subsection \ref{GeomStat}. 
The conditions rely on the interpretation of diffusion means as generalized Fr\'echet means, and are no more restrictive than in the case of the estimator 
\begin{equation*}
M_N=\argmin_{y\in \cM} \frac{1}{N}\sum_{i=1}^{N}\dist(y,x_i)^2
\end{equation*}
of the Fr\'echet mean set. For compact manifolds these conditions always hold. Next, we present sufficient conditions for the estimator to satisfy the general CLT of \citet{eltzner_smeary_2018} and consider the consequence on two spherical distributions. 
\begin{definition}
	Let $X$ be a random variable on $\cM$. Let $p$ be the minimal heat kernel on $\cM$. For samples $X_1,..., X_n\overset{\text{i.i.d.}}{\sim} X$ we define the \emph{sample log-likelihood function} $L_{t,n}:\cM \to \R$, 
	\begin{equation*}
	L_{t,n}(y) = - \frac{1}{n} \ln\Big(\prod_{i=1}^{n}p(X_i,y,t) \Big) = -\frac{1}{n} \sum_{i=1}^n \ln p(X_i,y,t)
	\end{equation*}
	for every $n\in\N$ and the \emph{sample diffusion $t$-mean sets} of $X_1,...,X_n$ are the sets
	\begin{equation*}
	E_{t,n} = \argmin_{y\in \cM} L_{t,n}(y).
	\end{equation*}
\end{definition}
Due to continuity of $p$ in the second component, every $E_{t,n}$ is a closed set. 
Considering the collection $\{E_{t,n}\}$ as random closed sets in the sense of 
\citet{Molch05}, they give rise to an M-estimator of the diffusion $t$-mean set $E_t$, which we call the \emph{diffusion estimator}. The following lemmas present sufficient conditions for the diffusion estimator to be strongly consistent.
\begin{lemma}
	Let $X:\Omega\to \cM$ be a random variable on $\cM$ and fix $t>0$. The diffusion estimator $E_{t,n}$ of $E_t(X)$ is strongly consistent in the sense of Ziezold (ZC) if either of the two following conditions hold:
	\begin{enumerate}
		\item $X$ has compact support.
		\item The map $y\mapsto \ln(p(x,y,t))$ is uniformly continuous and $\E[-\ln p (x,X,t)]<\infty$ for all $x\in\cM$. 
	\end{enumerate}
\end{lemma}
\begin{proof}
	This follows immediately from \citet[Theorem A.3]{huckemann_intrinsic_2011} by Remark \ref{RmkOnBounded}. 
\end{proof}
\begin{lemma}\label{BPC}
	Let $X:\Omega\to \cM$ be a random variable on $\cM$ and fix $t>0$. The diffusion estimator $E_{t,n}$ is strongly consistent in the sense of Bhattacharya and Patrangenaru (BPC) if $E_t\neq \emptyset$ and the following conditions hold:
	\begin{enumerate}
		\item $E_{t,n}$ satisfies (ZC)
		\item Every closed bounded subset of $\overline{\cup_{n=1}^\infty E_{t,n}}$ is compact (The Heine-Borel property)
		\item There exist $y_0\in \cM$ and $C>0$ such that $P(-\ln p(X,y_0,t)<C)>0$ and for every sequence $y_n\in \cM$ with $d(y_0,y_n)\to \infty$ there is a sequence $R_{t,n}\to\infty$ with $-\ln p(x,y_n,t)>R_{t,n}$ for all $x\in \cM$ with $-\ln p(x,y_0,t)<C$. 
	\end{enumerate}
\end{lemma}
\begin{proof}
	This follows directly from \citet[Theorem A.4]{huckemann_intrinsic_2011} by Remark \ref{RmkOnBounded}. 
\end{proof}

As for general results on manifolds, the situation simplifies when the manifold is compact. This is also the case for the previous two lemmas. In fact, the conditions are all satisfied in this case and can be dropped completely. 
\begin{cor}
	Fix $t>0$ and let $X$ be a random variable on $\cM$ and assume that $\cM$ is compact. Then, the diffusion estimator $E_{t,n}$ satisfies (ZC) and (BPC). 
\end{cor}
\begin{proof}
	Fix $t>0$. The estimator satisfies (ZC) because $\rho_t(x,y) = -\ln(p(x,y,t))\maxt$ is smooth on a compact domain, hence also uniformly continuous and bounded. For (BPC), the three assumptions 1.-3. of Lemma \ref{BPC} follow directly from $\cM$ being compact. Also, this ensures $E_t \neq \emptyset$. 
\end{proof} 
We will now move on from the topic of strong consistency to a generalized CLT for the diffusion estimator. Before we state the CLT for the diffusion estimator, we introduce some notation. Fix $t>0$ and let $X$ be a random variable with unique diffusion $t$-mean $\mu_{t}\in \cM$. For a measurable selection $\mu_{t,n}\in E_{t,n}$, define the sequence $x_{t,n}\in T_{\mu_t}\cM$ by
\begin{equation}\label{xtn}
x_{t,n} = \begin{cases}
\log_{\mu_t}(\mu_{t,n}) &\text{if } \mu_{t,n}\notin \mathcal{C}(\mu_t) \\
0 &\text{else}
\end{cases}.
\end{equation}
Define also $\widetilde{L}_t:\cD(\mu_t)\to \R$ by $\widetilde{L}_t(x)= L_t(\exp_{\mu_t}(x))$, which can be extended continuously to all of $T_{\mu_t}\cM$. Also, for each $y\in\cM$ we define the map $q_t(x,y)=-\ln p(\exp_{\mu_t}(x),y,t)$. 
\begin{theorem}[\textbf{General CLT}]
	Fix $t>0$ and let $X$ be a random variable on the manifold $\cM$ of dimension $m$. Assume that $\mu_t$ is the unique diffusion $t$-mean and that for every measurable selection $\mu_{t,n}$ of $E_{t,n}$ we have $\mu_{t,n}\overset{\mathbb{P}}{\to}\mu_t$, and define $x_{t,n}$ as in Equation \eqref{xtn}. Furthermore, assume that there exists $2\leq r \in \R$, a rotation matrix $R\in SO(m)$ and non-zero vectors $T_1,...,T_m$ such that 
	\begin{equation}\label{Taylor}
	\widetilde{L}(x) = \widetilde{L}_t(0) + \sum_{i=1}^m T_i |(Rx)_i|^r + o(||x||^r).
	\end{equation}
	Then for every measurable selection $\mu_{t,n}\in E_{t,n}$ it holds that
	\begin{equation}\label{dis}
	\sqrt{n}(V_1,...,V_n)^T \overset{\cD}{\to} \cN(0,\frac{1}{r^2}T^{-1}\emph{Cov}[\emph{grad}_xq_t(x,X)|_{x=0}]T^{-1})
	\end{equation}
	where $V_i = (Rx_n)_i|(Rx_n)_i|^{r-2}$ and $T=diag(T_1,...,T_m)$.
	\label{thm:clt}
\end{theorem}
\begin{proof}
	The result follows immediately from \citet[Theorem 11]{eltzner_smeary_2018} if all assumptions are met. Given the two assumptions above, the only remaining non-trivial condition is the almost surely locally Lipschitz and differentiability at the mean. Both properties hold since $p$ is smooth and thereby locally Lipschitz. 
\end{proof}
This Theorem is used as the basis for the definition of \emph{smeariness}. Here, we introduce the term for compact manifolds. 
\begin{definition}\label{SmearyDef}
	Fix $t>0$ and let $X$ be a random variable on the compact manifold $\cM$. If $E_t = \{\mu_t\}$ and the Taylor expansion of Equation \eqref{Taylor} exists and thus the limiting distribution as in Equation \eqref{dis} holds for every measurable selection $\mu_{t,n}$, then the diffusion estimator of $X$ is called \emph{$k$-th order smeary} or simply \emph{$k$-smeary} with $k=r-2$.
\end{definition}

As mentioned, several cases of smeariness have been discovered for the Fr\'echet mean. In the following examples we investigate the existence of smeariness for the diffusion means and also the influence of the time parameter $t$. We consider the order of smeariness as introduced in Definition \ref{SmearyDef} on the two pole distribution in Example \ref{ex:two_points1} and on a bimodal Brownian normal distribution in Example \ref{ex:bimodal-brownian}. In Section \ref{scn:application}, we consider the convergence rate of the diffusion estimator on a real data set. Lastly, in
Appendix \ref{sec:appendix-example}, we examine a previously studied distribution \citet{eltzner_smeary_2018}. 

\begin{example} \label{ex:two_points2}
	Let $X: \Omega \to \cS^m$ have the two pole distribution from Example \ref{ex:two_points1} and fix $t>\delta(m)$. As proved earlier, the point $\mu$ is the unique diffusion $t$-mean whenever $\alpha \leq \Lambda(t)$. In order to determine the asymptotic rate of the diffusion estimator we calculate the second derivative of $\widetilde{L}_{t,m}$ at $\delta = 0$. The second derivative 
	\begin{align*}
	\widetilde{L}_{t,m}''(\delta) &= \cos\delta \left( -\alpha \ell_{t,m}'(-\cos\delta) +(1-\alpha)\ell_{t,m}'(\cos\delta) \right) \\
	&\quad  +\sin\delta^2 \left( -\alpha \ell_{t,m}''(-\cos\delta) -(1-\alpha)\ell_{t,m}''(\cos\delta) \right)
	\end{align*}
	vanishes at $\delta=0$ exactly for $\alpha = \Lambda(t)$ and the third derivative vanishes for all values of $\alpha$. Using a Taylor expansion of $\widetilde{L}$ to achieve the expansion in Equation \eqref{Taylor}, this means that the regular CLT rate of convergence applies for $\alpha \leq \Lambda(t)$, whereas the estimator is at least $2$-smeary for $\alpha_t = \Lambda(t)$, if $\widetilde{L}_{t,m}$ still takes its global minimum at $\delta = 0$. According to Lemma \ref{loghdiff}, this is satisfied, if $t>\delta(m)$.
\end{example}

\begin{example}
	\label{ex:bimodal-brownian}
	In this example we consider the diffusion means of a bimodal \BND \space arising from the Brownian motion on $\cS^2$ visualized in Figure \ref{BNDdistribution}. We sample from a \BND \space on $\cS^2$ with variance $\sigma = 0.3$ in both poles $\mu$ and $-\mu$, with total probability mass $0.8$ in the north pole $\mu$ and $0.2$ in the south pole $-\mu$. 
	The estimated variance $V$ for the estimators of the Fr\'echet mean and the diffusion $t$-mean for $t = 0.4,0.6,1,2,4$ is plotted in Figure \ref{BNDestimates}. This variance is scaled by sample size $n$ and the dotted line represents the convergence rate of a 2-smeary estimator and the black line the standard CLT rate of $\sqrt{n}$. The weight $\alpha=0.2$ was chosen such that the Fr\'echet estimator is smeary to exemplify the behavior of the diffusion estimator in such a case. Figure \ref{BNDestimates} clearly indicates that the diffusion estimators have lower convergence rates and that an increase in $t$ leads to a faster converging diffusion estimator. 
\end{example}

\subsection{Small $t$ asymptotics} \label{shorttime}

For complete Riemannian manifolds, the uniform convergence \citet[Theorem 5.2.1]{hsu_stochastic_2002} on compact subsets 
\begin{equation} \label{eq:uniform-conv}
\lim_{t \to 0 } -2t\ln p(x,y,t) = \dist(x,y)^2
\end{equation}
connects the logarithmic heat kernel and the squared distance for arbitrary points $x,y\in \cM$. 
For fixed $t>0$, the scalar $2t$ does not affect the minima, so the sets $E_t(X)$ and $\argmin_{y\in \cM} \E [-2t\ln p(X,y,t)]$ coincide. For $t=0$ we have the set of classical Fr\'echet means
\begin{equation*}
M = \argmin_{y\in \cM}\E[\dist(X,y)^2].
\end{equation*}
We set
\begin{equation*}
  \begin{array}{rcll}
     \rho_t(x,y) &:=& -2t \log p(x,y,t)&\mbox{ and }F_t(y) := \E[\rho_t(X,y)]\mbox{ for }t>0\mbox{, as well as }\\
     \rho_0(x,y) &:=& \dist(x,y)^2&\mbox{ and }F_0(y) := \E[\rho_0(X,y)]\,,
  \end{array}
\end{equation*}
and assume that $\cM$ allows for an exhaustion by countably many compact sets $K_1\subseteq K_2 \subseteq \ldots\subseteq \cM$, $\bigcup_{j\in \N} K_j = \cM$ such that for every $y\in \cM$ there is $T(y)>0$ such that the following local tightness condition holds
\begin{equation}\label{eq:tightness}
	\mbox{ for every } \eps >0\mbox{ there is } j(\eps) \in \N \mbox{ so for all }j \geq j(\eps), t\in [0,T(y)] : \, \E[\rho_t(X,y)\cdot 1_{\{X \not \in K_j\}}] < \eps \,.
\end{equation}

\begin{theorem}\label{thm:diff-mean-frechet-mean}
  Assume that $\cM$ is a complete Riemannian manifold that allows for an exhaustion by countably many compact sets satisfying the local tightness condition (\ref{eq:tightness}), this is the case if $\cM$ is compact, then 
	\begin{equation*}
    M\neq \emptyset\,.
  \end{equation*}
	Further, for every sequence $0< t_k \to 0$ ($k\to \infty$) the sets of diffusion means converge to the set of classical Fr\'echet means
  in the sense of Ziezold, i.e.
  \begin{equation*}
    \mathop{\bigcap}_{n=1}^\infty \overline{\mathop{\bigcup}_{k=n}^\infty E_{t_k}} \subseteq M\,.
  \end{equation*}
  If additionally
  \begin{enumerate}
   	\item every closed bounded subset of $\overline{\cup_{k=1}^\infty E_{t_k}}$ is compact (the Heine-Borel property),
      \item there exist $y_0\in \cM$ and $C>0$ such that $P(\rho_t(X,y_0)<C)>0$ and for every sequence $y_n\in \cM$ with $d(y_0,y_n)\to \infty$ there is a sequence $R_n\to\infty$ with $\rho_t(x,y_0)>R_n$ for all $x\in \cM$ with $\rho_t(x,y_0)<C$ (coercivity),
  \end{enumerate}
  then convergence takes place in the sense of Bhattacharya and Patrangenaru, i.e. 
  for every $\eps>0$ there is a number $n=n(\eps,\omega)>0$ such that 
  \begin{equation*}
  \bigcup_{k=n}^\infty E_{t_k}\subset B(M,\eps)
  \end{equation*}
  where $B(M,\eps)$ is the ball of radius $\eps$ around $M$. 
\end{theorem}

\begin{proof}
    A self contained proof is given in Appendix \ref{subsec:thm47}.
\end{proof}

\subsection{Estimating $t$}
In previous sections we have discussed the estimation of the diffusion means for fixed time-parameter $t$ and it has been exemplified that increasing $t$ in terms of diffusion mean set and asymptotic rate of the diffusion estimator. In this section, we will discuss the estimation of $t$ for fixed mean.
 
A way to select an appropriate $t$ is to apply the same maximum likelihood method that was used to define the diffusion means. We define the \emph{diffusion variance set} $\mathcal{T}(y)$ at a point $y\in \cM$ to be
\begin{equation*}
\mathcal{T}(y) = \argmin_{t>0} L_t(y).
\end{equation*}
Letting $\mu_s$ be a diffusion $s$-mean for $s>0$, the derivative $\frac{d}{dt}L_t(\mu_s)$ can be written in terms of the Laplace Beltrami operator $\Delta$ due to $p$ being a solution to the heat equation, i.e. $\frac{d}{dt}p=\frac12\Delta p$:
\begin{align}
\frac{d}{dt} \int_\cM-\ln p(x,\mu_s,t)d\bP^X(x) &= -\int_\cM \frac{\frac{d}{dt}p(x,\mu_s,t)}{p(x,\mu_s,t)} d\bP^X(x)
&= -\int_\cM \frac{\frac12\Delta p(x,\mu_s,t)}{p(x,\mu_s,t)} d\bP^X(x).
\label{eq:lnheateq}
\end{align}
See also e.g. \citet{hsu_stochastic_2002}.
This means that the local minima of $L_t(\mu_s)$ w.r.t. $t$ can be found using the gradient descent algorithm
\begin{equation*}
t_{n+1} = t_n +\frac{\beta}2\int_\cM \frac{\Delta p(x,\mu_s,t)}{p(x,\mu_s,t)} d\bP^X(x)
\end{equation*} 
with learning rate $\beta >0$. As an illustration, we revisit the two pole distribution from Example \ref{ex:two_points1}.
\begin{example}\label{ex:two_points3}
	Let $X$ have the two-pole distribution as in Example \ref{ex:two_points1}.
	The Laplace Beltrami operator on $\cS^2$ with respect to the coordinates of Equation \eqref{coordinates} is 
	\begin{equation}
	\Delta_{\cS^2} f(\theta,\phi) = \frac{1}{\sin(\phi+\pi)}\frac{d}{d\phi}\left(\sin(\phi+\pi)\frac{d}{d\phi}f \right) + \frac{1}{\sin(\theta+\pi/2)^2}\frac{d^2}{d\theta^2}f.
	\end{equation}
	Fixing $\mu_t = \mu = (0,1,0)$ and recalling that the heat kernel only depends on $\phi = -\pi/2 + \delta $, we get an expression for the Laplace Beltrami operator of the heat kernel
	\begin{align*}
	\Delta_{\cS^2}p(\mu,y_\delta,t) &= \frac{1}{\sin(\pi/2+\delta)}\frac{d}{d\delta}\left(\sin(\pi/ 2+\delta)\frac{d}{d\delta}h_{t,2}(\cos\delta)\right) \\ 
	&= \sin\delta^2h_{t,2}''(\cos\delta) - \frac{\cos\delta^2-\sin\delta^2}{\cos\delta}h'_{t,2}(\cos\delta).
	\end{align*}
	Thus, for a fixed $\alpha\in [0,\frac12]$ we have
	\begin{align*}
	\frac{d}{dt}L_t(\mu) &= -\alpha \frac{h'_{t,2}(-1)}{h_{t,2}(-1)} + (1-\alpha)\frac{h'_{t,2}(1)}{h_{t,2}(1)} = -\alpha \ell'_{t,2}(-1) + (1-\alpha) \ell'_{t,2}(1).
	\end{align*}
	We now consider $\alpha = \Lambda(s)$ for some $s>0$. We then recognize this expression from the derivative of the log-likelihood function in Equation \eqref{Gder} and conclude that $\frac{d}{dt}L_t(\mu)|_{t=s} = 0$, the value which leads to a smeary $\mu$. The derivative $\frac{d}{dt}L_t(\mu)|_{t = r}$ is positive for $r>s$ and negative for $r<s$, which indicates a local minimum such that the diffusion variance set at $\mu$ is $\mathcal{T}(\mu) = \{s\}$.
	
	Joint estimation of $\mu$ and $t$ can therefore only lead to smeariness for $t=s$ or non-uniqueness for $t<s$ of the diffusion mean. Joint estimation of parameters is treated numerically in Figure \ref{2poletestimate} below.
\end{example}

\subsection{Removing smeariness: Joint estimation of $\mu$ and $t$}
Analogous to the Gaussian MLE in Euclidean space, it is possible to estimate $\mu_t$ and $t$ jointly, to get
\begin{equation}
 (\mu_{n}, t_{n}) = \mathop{\argmin}_{t>0, \mu\in\cM} L_{t,n}(\mu).
 \label{eq:joint_est}
\end{equation}
Such an estimation was performed on a data set on $S^1$ by 
\citet{hansen_diffusion_2021}. It was found that the diffusion mean $\mu_{n}$ jointly estimated with $t_n$ exhibits a considerably lower magnitude of finite sample smeariness (introduced in Section \ref{scn:application}) than the Fr\'echet mean.

In fact, it turns out that non-smeariness of diffusion means with jointly estimated $\mu$ and $t$ is a very general phenomenon as shown in the following result. We state it for the population mean, but the result holds similarly for sample estimators minimizing \eqref{eq:joint_est}.
\begin{theorem} \label{theo:joint_estimation}
Assume $\mu,t$ minimizes the likelihood \eqref{negloglikelihood}. Then
\begin{equation*}
\Delta_{\mu}
\E[
-\ln p(X,\mu,t) 
]
=\E[
\left\|\nabla_x\ln p(x,\mu,t)_{x=X}\right\|^2
]
.
\end{equation*}
\end{theorem}
\begin{proof}
Since $p$ is a solution to the heat equation, we can continue \eqref{eq:lnheateq} to see that $-\ln p$ satisfies 
\begin{align*}
&\frac{d}{dt} \int_\cM-\ln p(x,\mu,t)d\bP^X(x) 
= -\int_\cM \frac{\frac12\Delta_{x} p(x,\mu,t)}{p(x,\mu,t)} d\bP^X(x)
\\
&\qquad 
= -\int_\cM 
\frac12\Delta_{x}\ln p(x,\mu,t) 
+
\frac12\left\|\nabla_x\ln p(x,\mu,t)_{x=X}\right\|^2
d\bP^X(x).
\end{align*}
Because of symmetry of the heat kernel, $\Delta_\mu\ln p(x,\mu,t)=\Delta_x\ln p(x,\mu,t)$.
At a critical point for $t$, the left-hand side vanishes and the result follows.
\end{proof}
As a consequence of this result, the expected value of the Hessian trace of the negative log-likelihood is non-zero if $\left\|\nabla_x\ln p(x,\mu,t)_{x=X}\right\|^2$ is non-zero 
with positive probability. This happens unless $X$ is only supported on the critical points of the log-likelihood, as in the two pole distribution from Examples \ref{ex:two_points1}, \ref{ex:two_points2}, \ref{ex:two_points3}. This result does not imply that all eigenvalues of the expected Hessian are strictly positive but it means that some are and the rest could be 0 because negative values cannot happen at a local minimum. Positive eigenvalues of the expected Hessian correspond to $r=2$ in the Taylor expansion in Theorem~\ref{thm:clt}. The phenomenon where some, but not all eigenvalues are zero, i.e smeariness occurs in those directions only, has been introduced by \citet{eltzner_geometrical_2020,TranEltznerHuckemann2021} as \emph{directional smeariness}. These consideration yield at once the following.

\begin{cor}
\label{cor:nonsmeariness}
Let $X_1,\ldots,X_n$ be an i.i.d. random sample on $M$. If for every $\mu$ and $t$ one has
$\sum_{j=1}^n \left\|\nabla_x\ln p(x,\mu,t)|_{x=X_j}\right\|^2>0$, then $\mu$ is at most directionally smeary.
\end{cor}

We here briefly discuss the generality of Corollary~\ref{cor:nonsmeariness} and possible ways to strengthen the result.
On real analytic manifolds, the heat kernel is real analytic 
so that $-\ln p(x,\mu,t)$ can only have isolated critical points. This happens for the two-pole example as mentioned above, but it is "rare" in practice. E.g. if the data distribution is non-singular, it is almost sure not the case. 

The Hessian trace being positive does not rule out directional smeariness. In order to strengthen this result, we need to consider the second order derivative terms in the Laplacian individually. One possibility is here to use particular a structure of $M$, e.g. if $M$ has a product Riemannian structure, the Hessian trace will generally be positive on each of these factors individually. Even more generally, with the anisotropic Brownian-like flows considered in \citet{sommer_modelling_2016}, the diffusion variance can be fitted by maximum likelihood in all directions individually, corresponding to estimation of a full covariance matrix. This provides additional information on the directional derivative terms in the Laplacian individually, and we expect to be able to strengthen Corollary~\ref{cor:nonsmeariness} in this setting in future work.

We perform a numerical study of joint parameter estimation for the two pole distribution discussed in Example \ref{ex:two_points3} and the Brownian normal distribution. The results are displayed in Figure \ref{fig:joint_estimation}.
	
\begin{figure}[ht]
	\begin{subfigure}[t]{.45\linewidth}
		\centering
		\includegraphics[width=\linewidth]{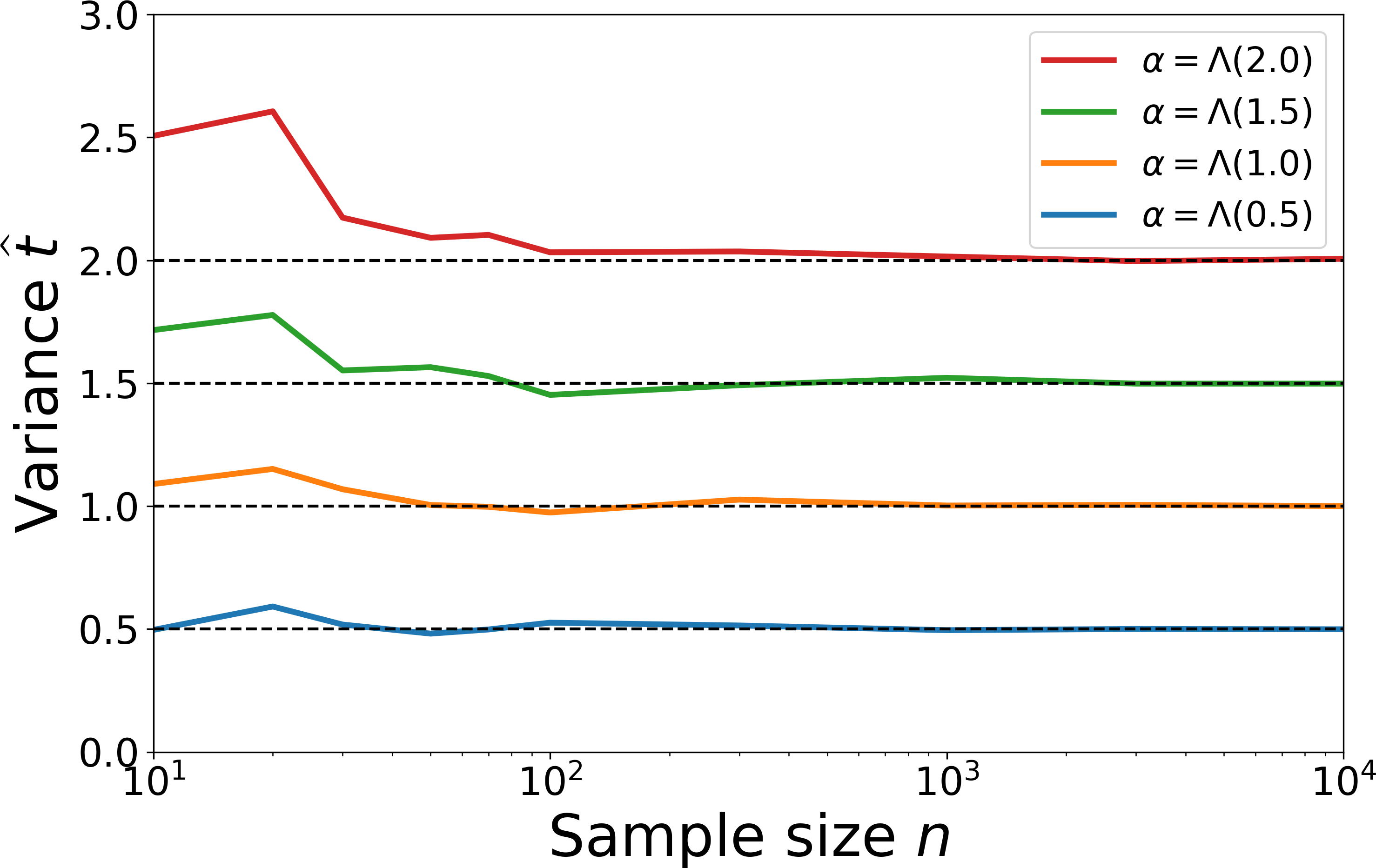}
		\caption{Estimates of $t$ for the two pole distribution where (blue) $\alpha = \Lambda(0.5)$, (orange) $\alpha = \Lambda(1.0)$, (green) $\alpha = \Lambda(1.5)$, and (red) $\alpha = \Lambda(2.0)$, see Example \ref{ex:two_points1}.}
		\label{2poletestimate}
	\end{subfigure}
	\hspace*{0.02\linewidth}
	\begin{subfigure}[t]{.45\linewidth}
		\centering
		\includegraphics[width=\linewidth]{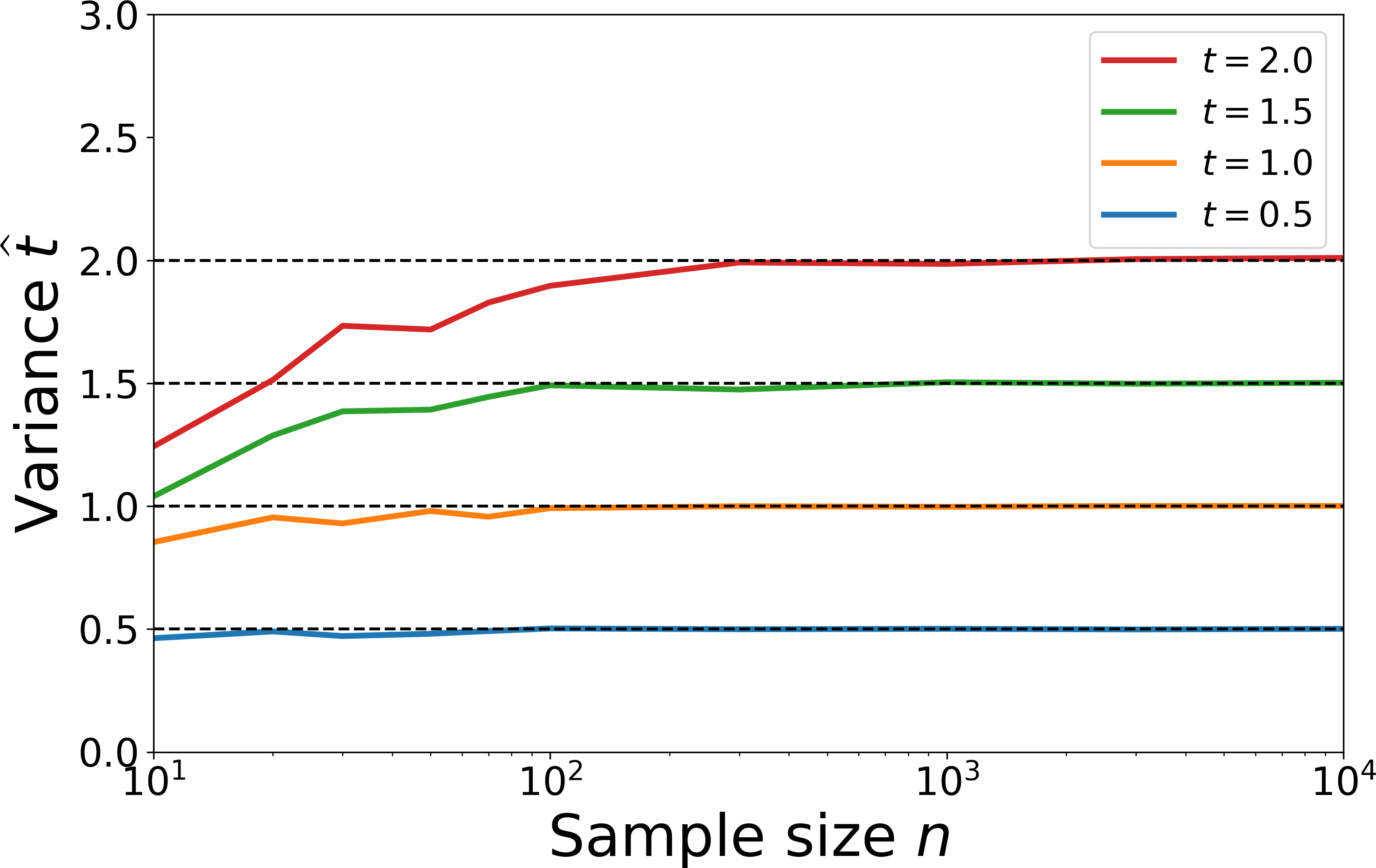}
		\caption{Estimates of $\widehat{t}$ for the Brownian normal distribution with $t = 0.5,1,1.5,2$.}
		\label{BrownianNormalTestimate}
	\end{subfigure}
	\caption{The plots display estimates of the diffusion variance $\widehat{t}$. The left plot show estimates on the two-pole distribution for various south pole weights $\alpha$ (yielding smeariness) and the right panel shows estimates from a \BND\ for various diffusion variances $t$.\label{fig:joint_estimation}}
\end{figure}

For the Brownian normal distribution we recover the true variance, as displayed in Figure \ref{BrownianNormalTestimate}, which illustrates asymptotic consistency of the estimator. For the two-pole distribution, the variance $t$ estimated jointly with the mean $\mu$ is converges the variance $t=s$ leading to smeariness, as seen in Figure \ref{2poletestimate}. This remarkable result does not contradict Corollary \ref{cor:nonsmeariness}, since the two pole distribution features a critical point for $\nabla_x\ln p(x,\mu,t)$ for $x$ being the south pole and thus violates the assumptions of the corollary. While the fact that the mean, if jointly estimated with the diffusion variance, is smeary in the two-pole model is undesirable, the non-uniqueness of the Fr\'echet mean in the same model is arguable even worse. Manually fixing a higher $t$ than estimated, leads to a more benign, non-smeary asymptotic behavior of $\mu$. This illustrates that diffusion means can have preferable asymptotic behavior to the Fr\'echet mean even in pathological cases.

\section{Application}
\subsection{Reducing Finite Sample Smeariness (FSS) in Geomagnetic Data}\label{scn:application}

In contrast to smeariness, i.e. that the variance of estimators $\widehat \mu_n$ based on a sample $X_1,\ldots,X_n$ of a population mean $\mu$ scales not asymptotically with inverse sample size $1/n$, there are distributions where this happens within finite sample size regime only. This phenomenon has been called \emph{finite sample smeariness} (FSS) by \citet{hundrieser_finite_2020} and is measured, given an underlying metric $d(\cdot,\cdot)$, by the \emph{variance modulation} that has been introduced by \citet{pennec_curvature_2019}
	\begin{align*}
    \mathfrak{m}_n := \frac{n\mathbb E[d(\widehat \mu_n, \mu)^2]}{\sum_{j=1}^n d(X_j,\mu)^2}\,.
\end{align*}
While $\mathfrak{m}_n = 1$ is constant in $n$ in case of Euclidean spaces, in Figures \ref{BNDestimates}, \ref{MagneticPolePlot} and \ref{MagneticPolePlot_sim} the estimators of the curves $n\mapsto \mathfrak{m}_n$ initially increase as in case of smeariness and settle (sometimes after a moderate decrease) at a constant larger than one. Presence of FSS, i.e. of $\mathfrak{m}_n > 1$ makes, among others, quantile based tests inapplicable, in contrast to bootstrap tests, having lowered power, though, see \citet{hundrieser_finite_2020}. This effect is stronger the higher the initial overshoot and the higher the asymptotic horizontal.

Spherical statistics has a natural application to geolocation data, as they can be considered as observed random variables on the $2$-sphere $\cS^2$. To show the occurrence of finite sample smeariness, as described by \citet{hundrieser_finite_2020}, on real data, we here include an example of estimating both the Fr\'echet mean and diffusion $t$-mean for $t = 1$ and $t=2$ on data sets of magnetic pole locations during geomagnetic pole reversals presented in \citet{McEl96} and downloadable from \href{ftp://ftp.ngdc.noaa.gov/geomag/Paleomag/access/ver3.5}{ftp://ftp.ngdc.noaa.gov/geomag/Paleomag/access/ver3.5}, where each curve in the panels of Figure \ref{MagneticPolePlot} represent a separate data set. These data sets were investigated by \citet{eltzner_geometrical_2020}, who found finite sample smeariness of the Fr\'echet mean. The distribution of the magnetic pole locations during the reversals is similar to the bimodal \BND \space of Example \ref{ex:bimodal-brownian}, typically with lower variance. The limit of zero variance corresponds to the two pole distribution of Example \ref{ex:two_points1}.

\begin{figure}[ht]
	\centering
	\vspace{0pt}
	\begin{subfigure}[t]{.32\linewidth}
		\centering
		\includegraphics[width=\linewidth]{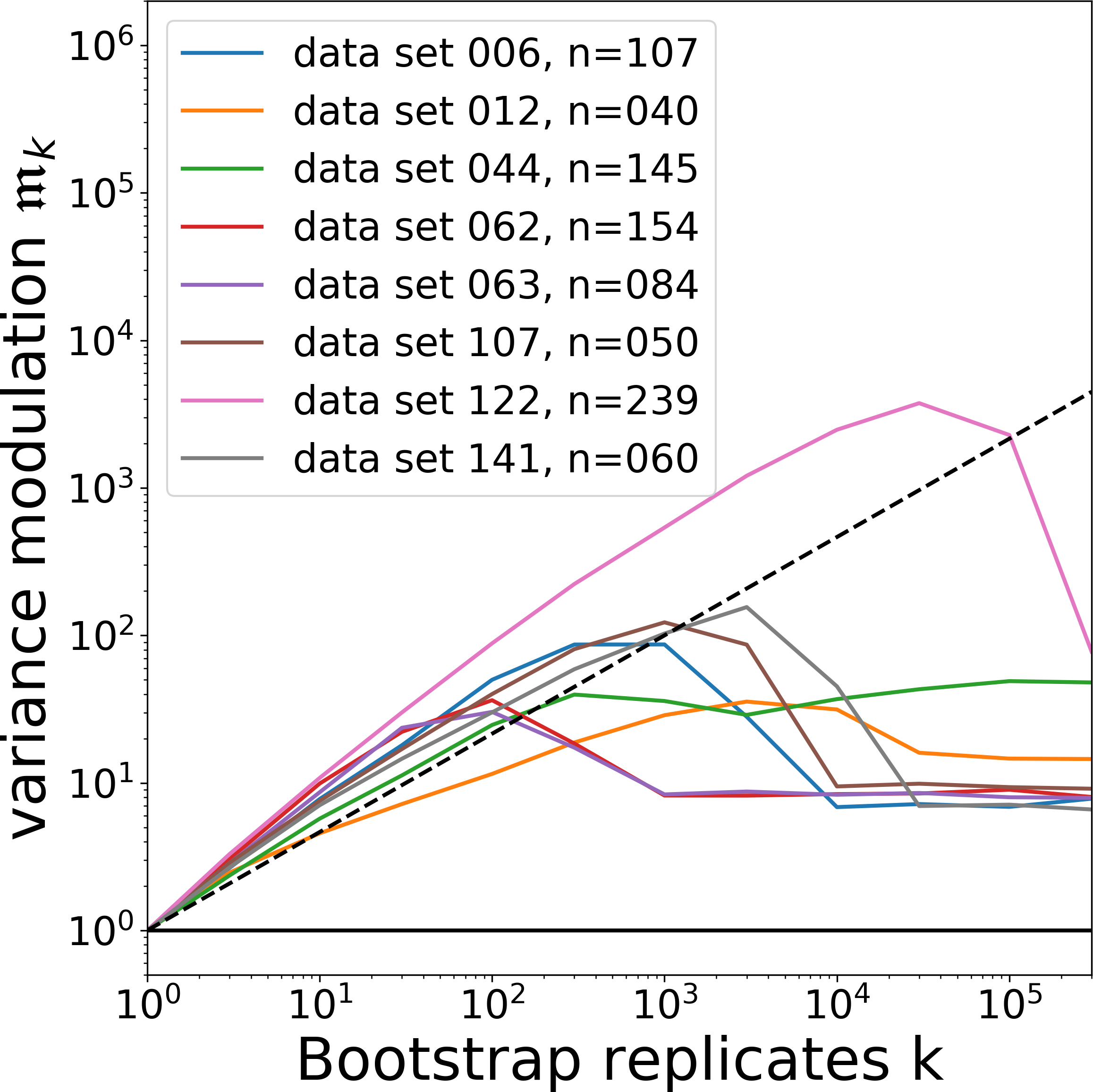}
		\caption{Fr\'echet mean}
		\label{fig:MagneticPolePlot-sub1}
	\end{subfigure}\hfill
	\begin{subfigure}[t]{.32\linewidth}
		\centering
		\includegraphics[width=\linewidth]{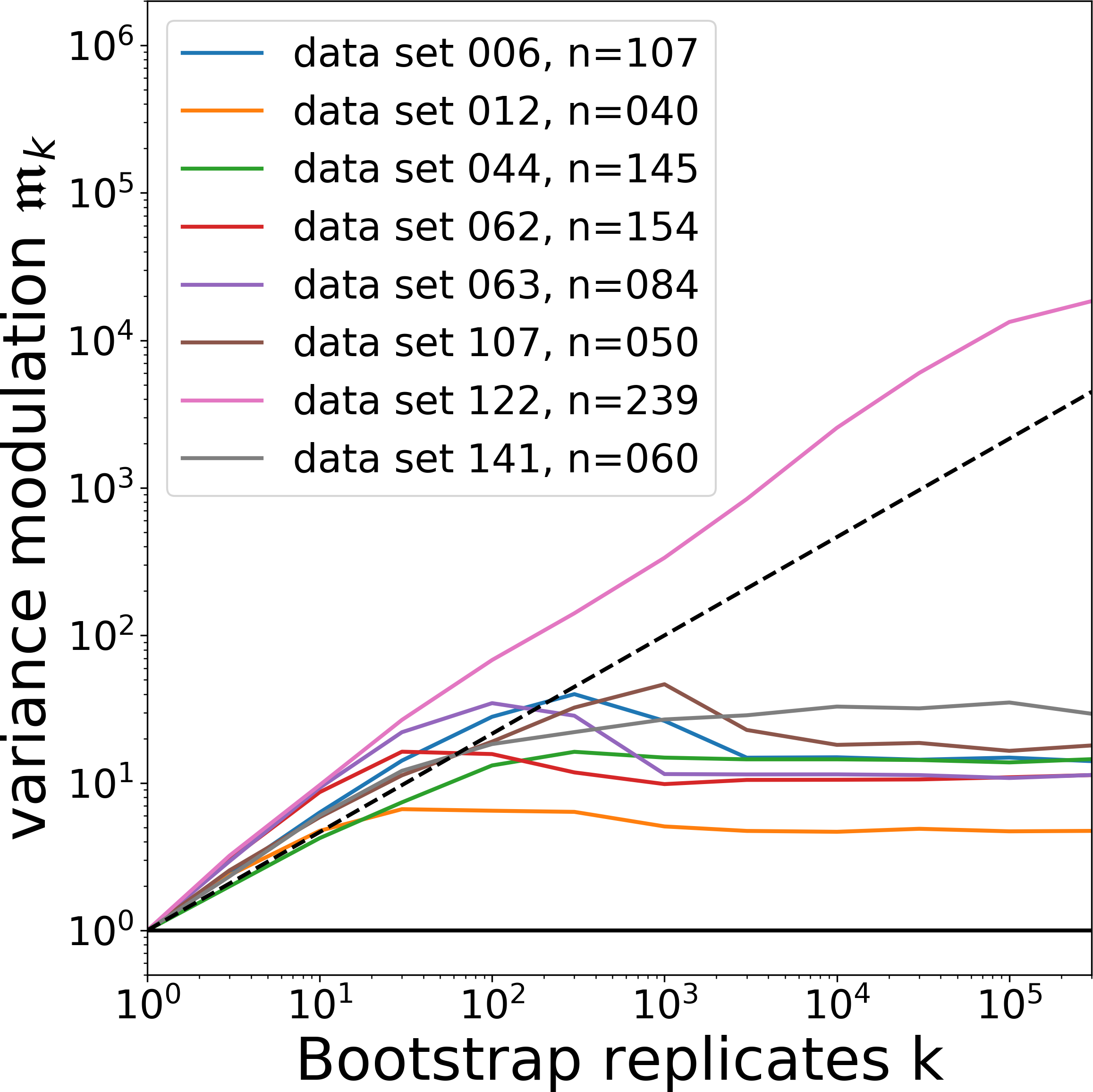}
		\caption{$t = 1$ diffusion mean}
		\label{fig:MagneticPolePlot-sub2}
	\end{subfigure}\hfill
	\begin{subfigure}[t]{.32\linewidth}
		\centering
		\includegraphics[width=\linewidth]{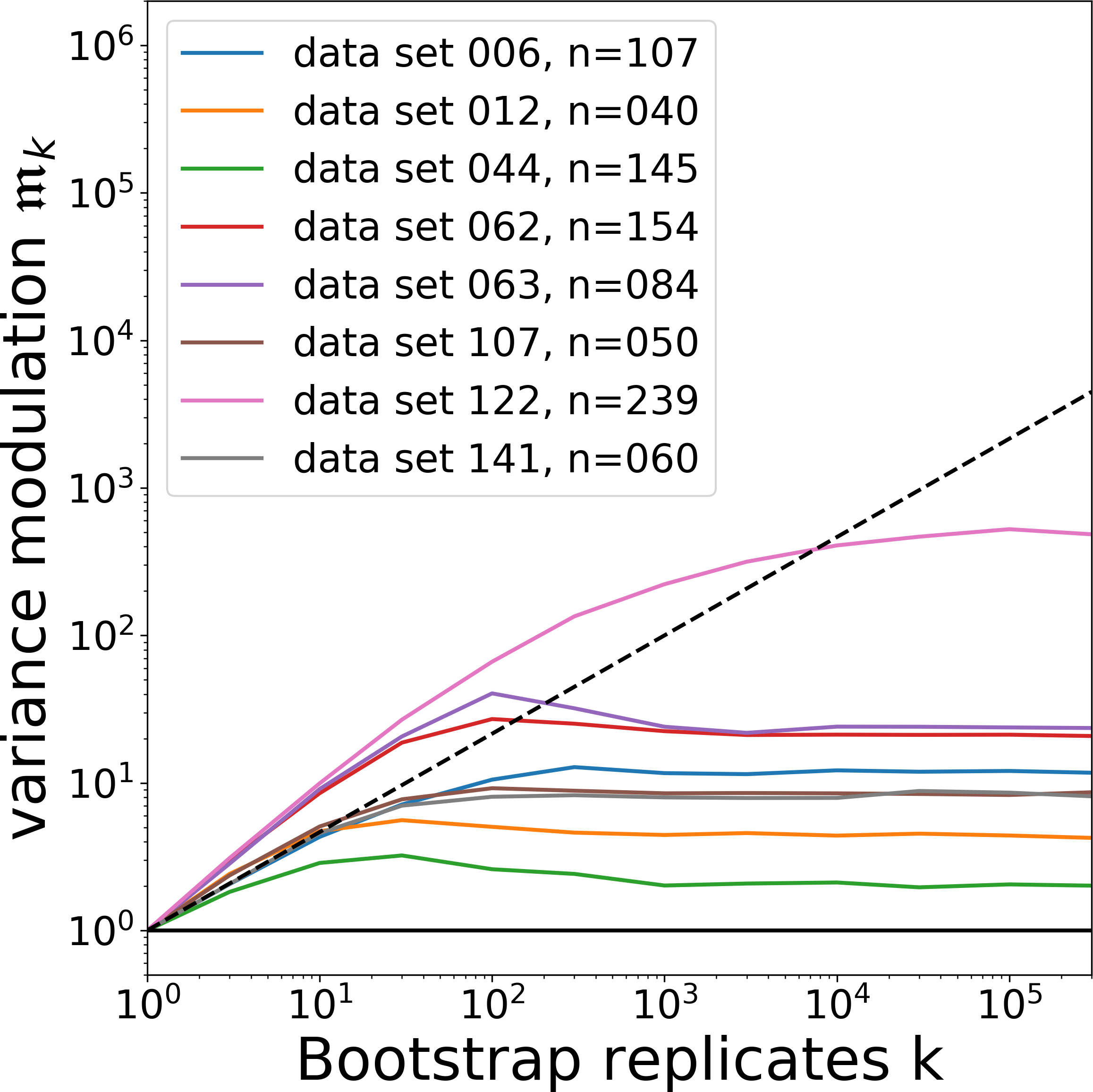}
		\caption{$t = 2$ diffusion mean}
		\label{fig:MagneticPolePlot-sub3}
	\end{subfigure}\hfill
	\caption{Bootstrap variance modulation curves (\ref{eq:bootstrap-modulation}) 
	of the estimated Fr\'echet mean and diffusion mean for $t=1$ and $t=2$ of eight magnetic pole data sets from \citet{McEl96}, each represented by a different color. The dashed line indicates $2$-smeariness while the solid black horizontal indicates $0$-smeariness. }
	\label{MagneticPolePlot}
\end{figure}

Figure \ref{MagneticPolePlot} displays the $k$-out-of-$n$ bootstrap variance modulations as introduced in \citet{eltzner_geometrical_2020,hundrieser_finite_2020}
\begin{align}\label{eq:bootstrap-modulation}
    \mathfrak{m}_k := \frac{k\mathrm{Var}[\mu_k^*]}{\mathrm{Var}[\{X_j\}_{j=1}^n]}
\end{align}
of the Fr\'echet mean and diffusion mean for $t = 1$ and $t = 2$ of eight different data sets. It is evident from the plots that the choice of $t$ has an effect on FSS, which is consistent with the results from the previous examples: By increasing $t$, asymptotic variance is reduced and initial, partly strong, overshoots are eventually removed, i.e. effects of FSS are reduced by increasing $t$.

Lastly, we jointly estimate the diffusion mean and diffusion variance for the data sets of magnetic pole locations during geomagnetic pole reversals \citet{McEl96}, which yields the results displayed in Figure~\ref{MagneticPolePlot_sim}. The results indicate a lower maximal variance modulation for the diffusion mean than for the Fr\'echet mean. Above bootstrap sample sizes of $k=100$ all curves appear compatible with the standard CLT rate. The modulation curves for the diffusion means appear close to monotonically growing. In the notation introduced by \citet{hundrieser_finite_2020} this means that there is no Type II FSS. The estimated $t$ parameters lie approximately in the range $[1.3,3.1]$, with $t$-estimates being smaller for the data sets with faster convergence to the CLT rate.

\begin{figure}[ht]
	\centering
	\begin{subfigure}[t]{.42\linewidth}
		\centering
		\includegraphics[width=\linewidth]{plot_diff_means_smeary_magnets_0}
		\caption{Fr\'echet mean}
		\label{fig:magnets-sub1}
	\end{subfigure}
	\hspace*{0.05\linewidth}
	\begin{subfigure}[t]{.42\linewidth}
		\centering
		\includegraphics[width=\linewidth]{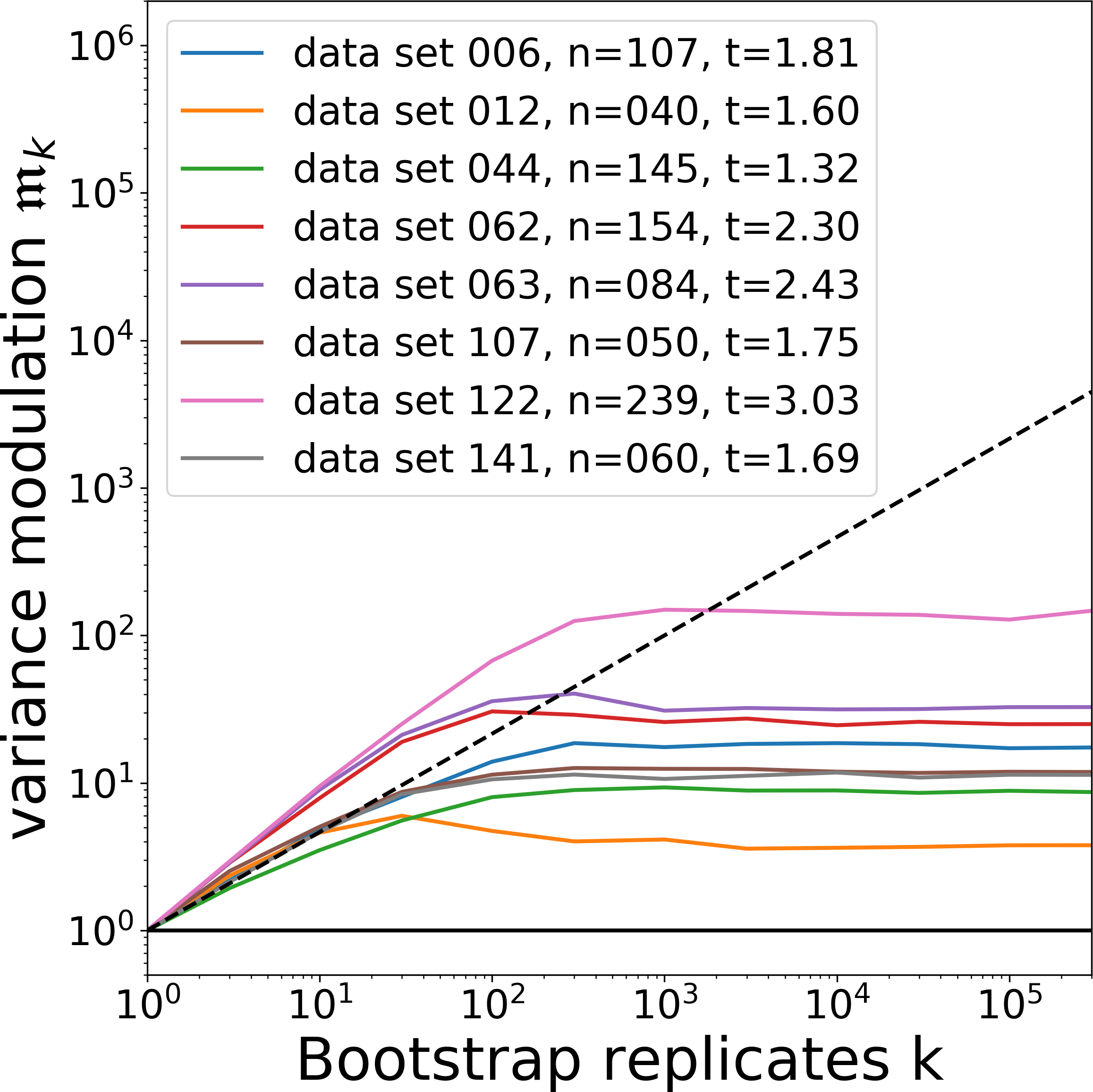}
		\caption{diffusion mean with fitted $t$}
		\label{fig:magnets-sub2}
	\end{subfigure}\hfill
	\caption{Bootstrap variance modulation curves (\ref{eq:bootstrap-modulation}) of the estimated Fr\'echet mean and diffusion mean for optimal $t$-values of eight magnetic pole data sets from \citet{McEl96}, each represented by a different color. The dashed line indicates $2$-smeariness while the solid black horizontal indicates Euclidean $0$-smeariness.} 
	\label{MagneticPolePlot_sim}
\end{figure}

\begin{table}
\caption{Variances of $n$-out-of-$n$ bootstrap means for the Fr\'echet mean (3rd column) and the diffusion means (4th column), with original sample size $n$ (2nd column) of the corresponding data sets (1st column), and the ratio diffusion mean variance over Fr\'echet mean variance (5th column) where the mean and diffusion variance $t$ (6th column) have been estimated jointly. 
\label{tab:variances}}
  \begin{center}
    \begin{tabular}{l|c|c|c|c|c}
      Data set & sample size $n$ & Fr\'echet mean variance & diffusion mean variance & ratio & $t$\\
      \hline
      006 & 107 & 1.118 & 0.323 & 0.289 & 1.809\\
      012 &  40 & 0.403 & 0.281 & 0.697 & 1.601\\
      044 & 145 & 0.394 & 0.119 & 0.302 & 1.322\\
      062 & 154 & 0.482 & 0.487 & 1.010 & 2.303\\
      063 &  84 & 1.127 & 1.049 & 0.931 & 2.429\\
      107 &  50 & 1.023 & 0.473 & 0.462 & 1.754\\
      122 & 239 & 1.932 & 1.173 & 0.607 & 3.031\\
      141 &  60 & 0.824 & 0.406 & 0.493 & 1.685
    \end{tabular}
  \end{center}
\end{table}
\begin{figure}[ht]
	\centering
	\begin{subfigure}[t]{.33\linewidth}
		\centering
		\includegraphics[width=\linewidth]{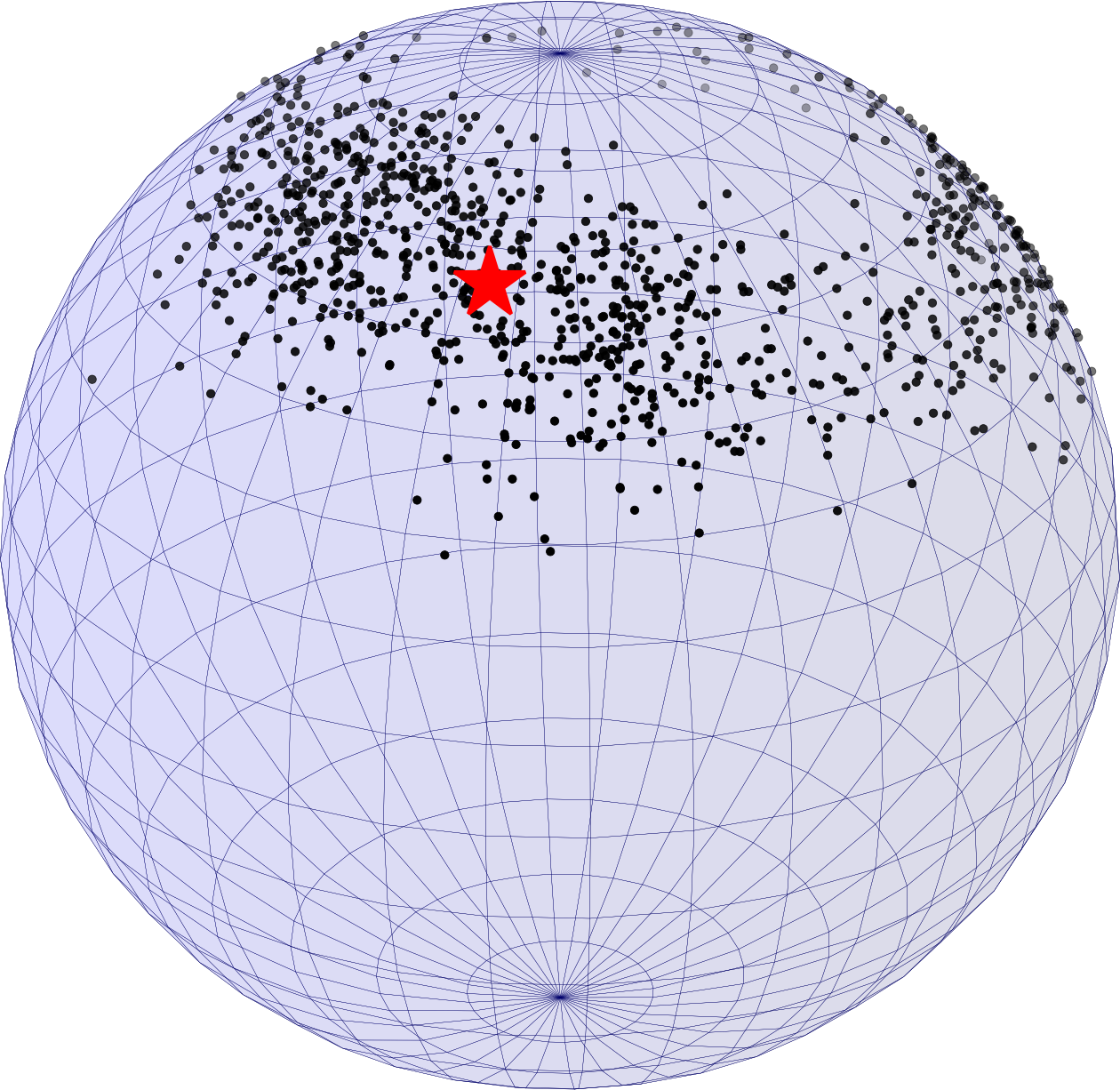}
		\caption{Data set 044 Fr\'echet means}
	\end{subfigure}
	\hspace*{0.05\linewidth}
	\begin{subfigure}[t]{.33\linewidth}
		\centering
		\includegraphics[width=\linewidth]{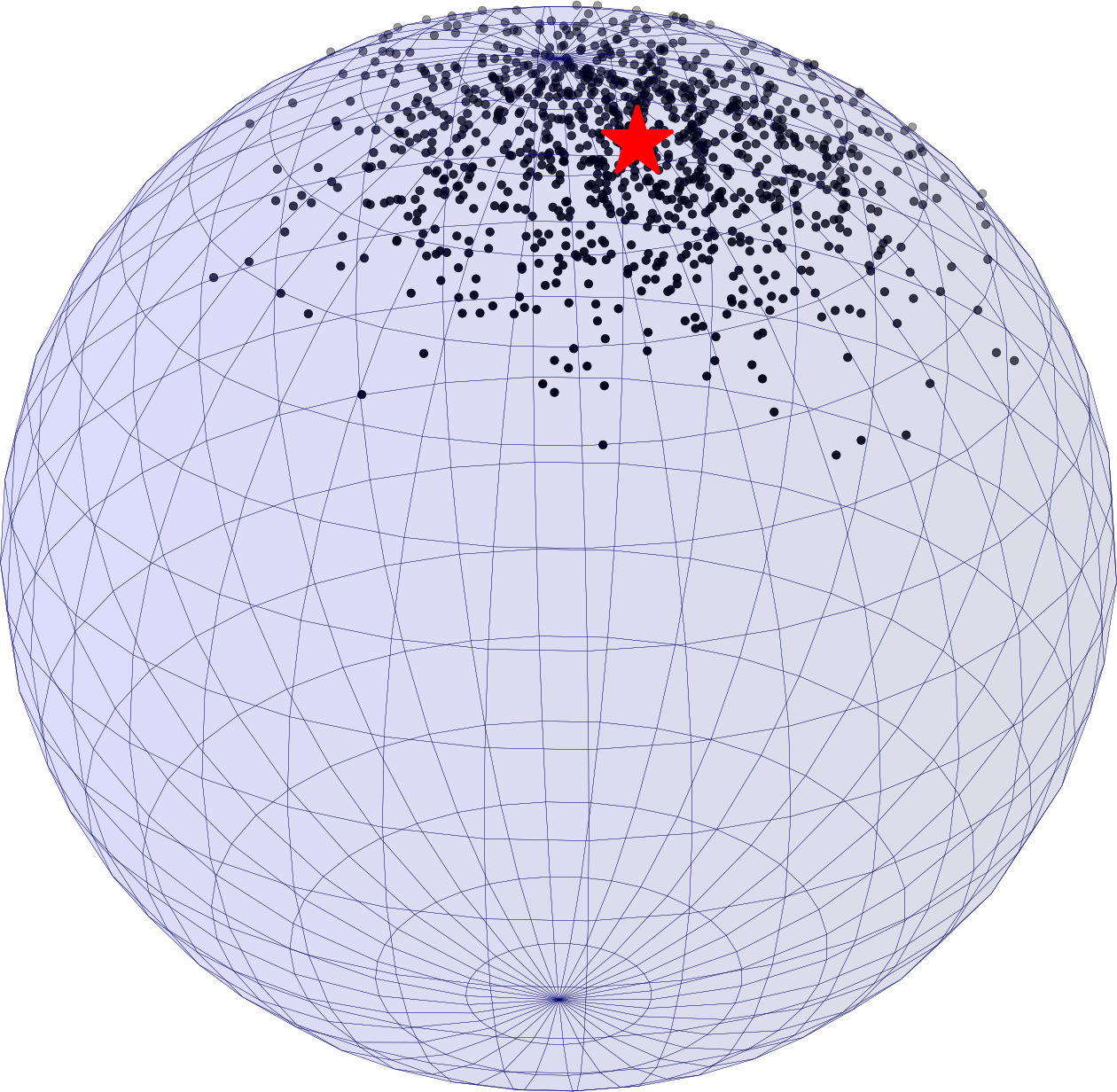}
		\caption{Data set 044 diffusion means}
		\label{fig:bootstrap_vars-sub2}
	\end{subfigure}\hfill\\
	\begin{subfigure}[t]{.33\linewidth}
		\centering
		\includegraphics[width=\linewidth]{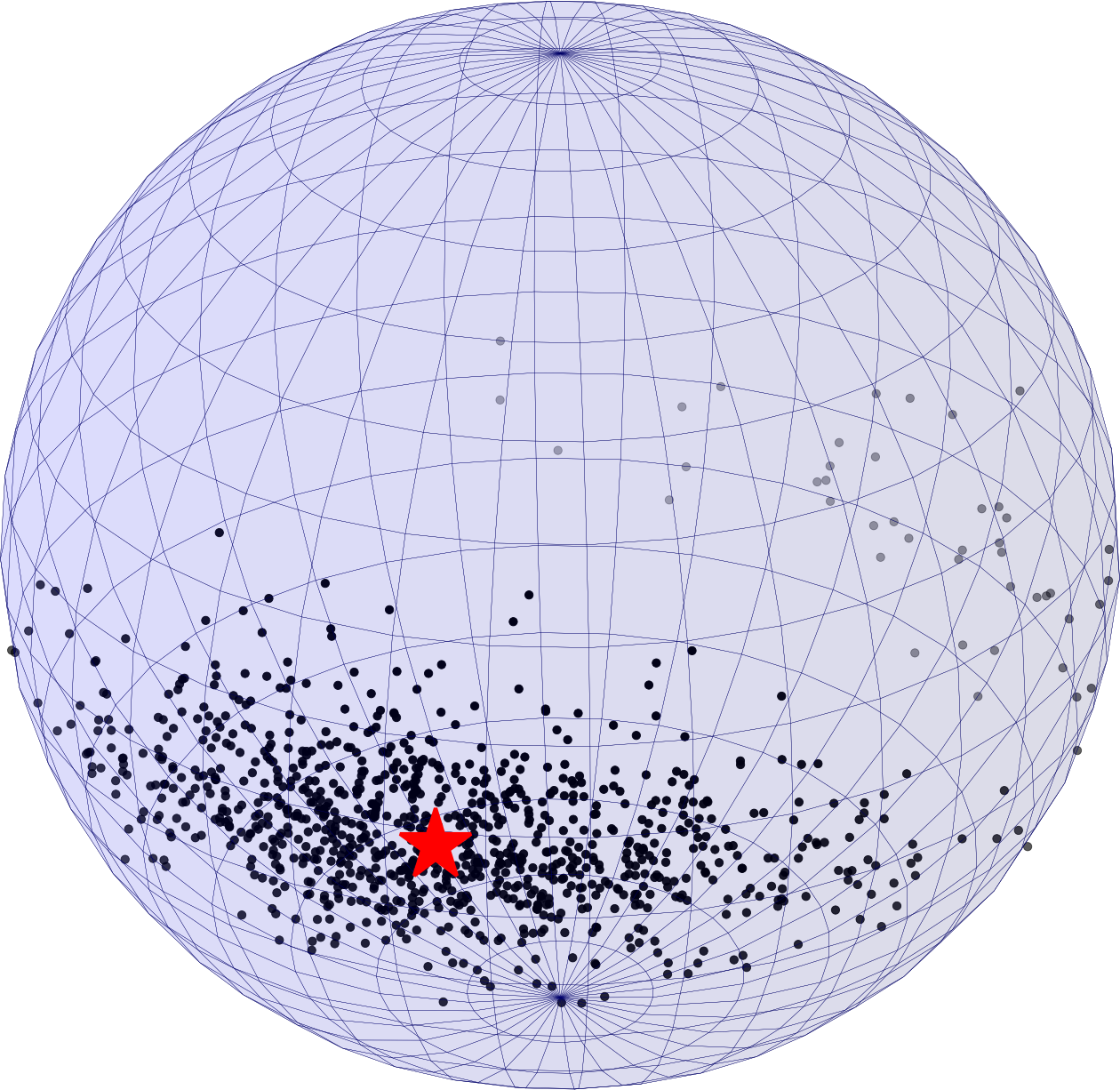}
		\caption{Data set 062 Fr\'echet means}
		\label{fig:bootstrap_vars-sub3}
	\end{subfigure}
	\hspace*{0.05\linewidth}
	\begin{subfigure}[t]{.33\linewidth}
		\centering
		\includegraphics[width=\linewidth]{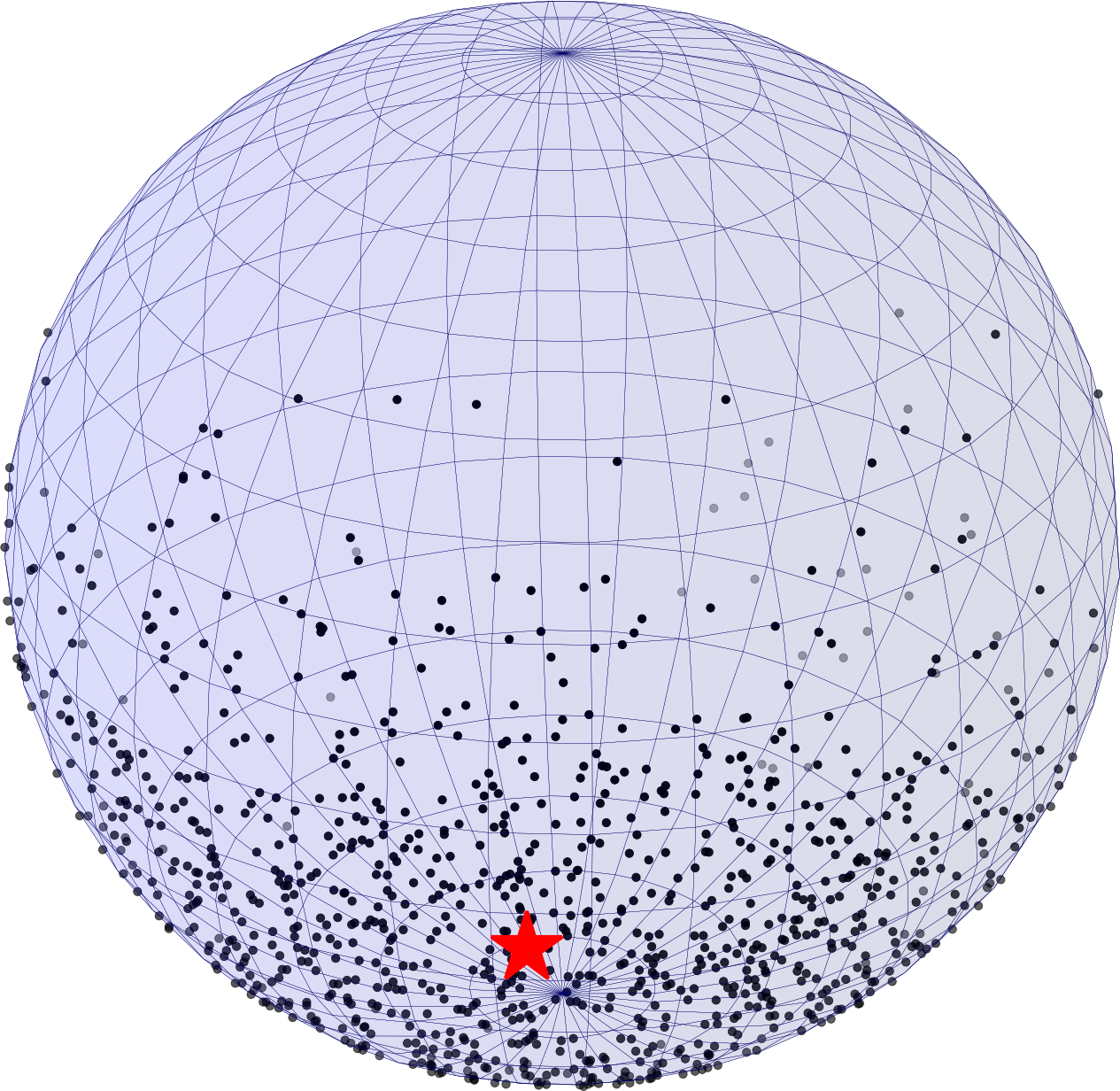}
		\caption{Data set 062 diffusion means}
		\label{fig:bootstrap_vars-sub4}
	\end{subfigure}\hfill
	\caption{Bootstrap means (black dots) for data sets 044, where diffusion means have a much lower variance than Fr\'echet means, and 062, where variances are comparable. The red stars indicate the sample means. Panel~(a) shows that the bootstrapped Fr\'echet means are near a small circle, which is consistent with the circular mode of the sample mean distribution in case of smeariness. In contrast, the bootstrapped diffusion means displayed in Panel~(b) feature a fairly concentrated normal distribution consistent with a non-smeary mean. In Panel~(c) the distribution of bootstrapped Fr\'echet means is again very anisotropic, featuring a tail of almost antipodal outliers, indicating smeariness. While the bootstrapped diffusion means displayed in Panel~(d) have a similar high overall variance, the distribution is rather isotropic around the sample mean, indicating isotropic FSS.}
	\label{fig:bootstrap_vars}
\end{figure}

In Table \ref{tab:variances} we compare the variances of $n$-out-of-$n$ bootstrap means for the Fr\'echet mean and the diffusion mean with jointly estimated diffusion variance. These are relevant for the power of hypothesis tests, smaller variance translating to higher power. The diffusion mean variance is clearly smaller for all data sets except data sets 062 and 063, where both variances are of similar size. This indicates that in case of a finite sample smeary Fr\'echet means, diffusion means achieve higher power in hypothesis tests.

In Figure \ref{fig:bootstrap_vars} we illustrate the bootstrapped samples for data sets 044, for which the diffusion mean has much lower variance than the Fr\'echet mean, and 062, for which variances are similar. One can see that the distribution of diffusion means appear much more isotropic around the sample mean.

These results indicate that a joint estimation of $\mu$ and $t$ generally leads to a diffusion mean $\mu_{n}$ which exhibits considerably lower magnitude of FSS than the Fr\'echet mean. 

\subsection{Power of Hypothesis Tests}

To see that the reduction in FSS is of practical importance, we compare the results of hypothesis tests for the pole position on the $8$ geomagnetism data sets discussed above and on the wind direction data set discussed in \cite{hundrieser_finite_2020}. We perform the two sample $T^2$ tests for heteroscedastic data using bootstrapped covariances and quantiles designed to achieve high power as discussed in detail in \cite{eltzner_hyptest_2017}. We display the resulting sets of p-values in terms of pp-plots in Figure~\ref{fig:tests}.

For the wind direction data, we considered data sets for the two cities separately and only took into account tests between data sets of which at least one displayed considerable FSS with a variance modulation of at least $5$. This was the case for $3$ out of $20$ data sets for Basel, leading to $54$ tests considered and for $5$ out of $20$ data sets for G\"ottingen, leading to $85$ tests considered, for a total number of $139$ hypothesis tests.

\begin{figure}[ht]
	\centering
	\begin{subfigure}[t]{.42\linewidth}
		\centering
		\includegraphics[width=\linewidth]{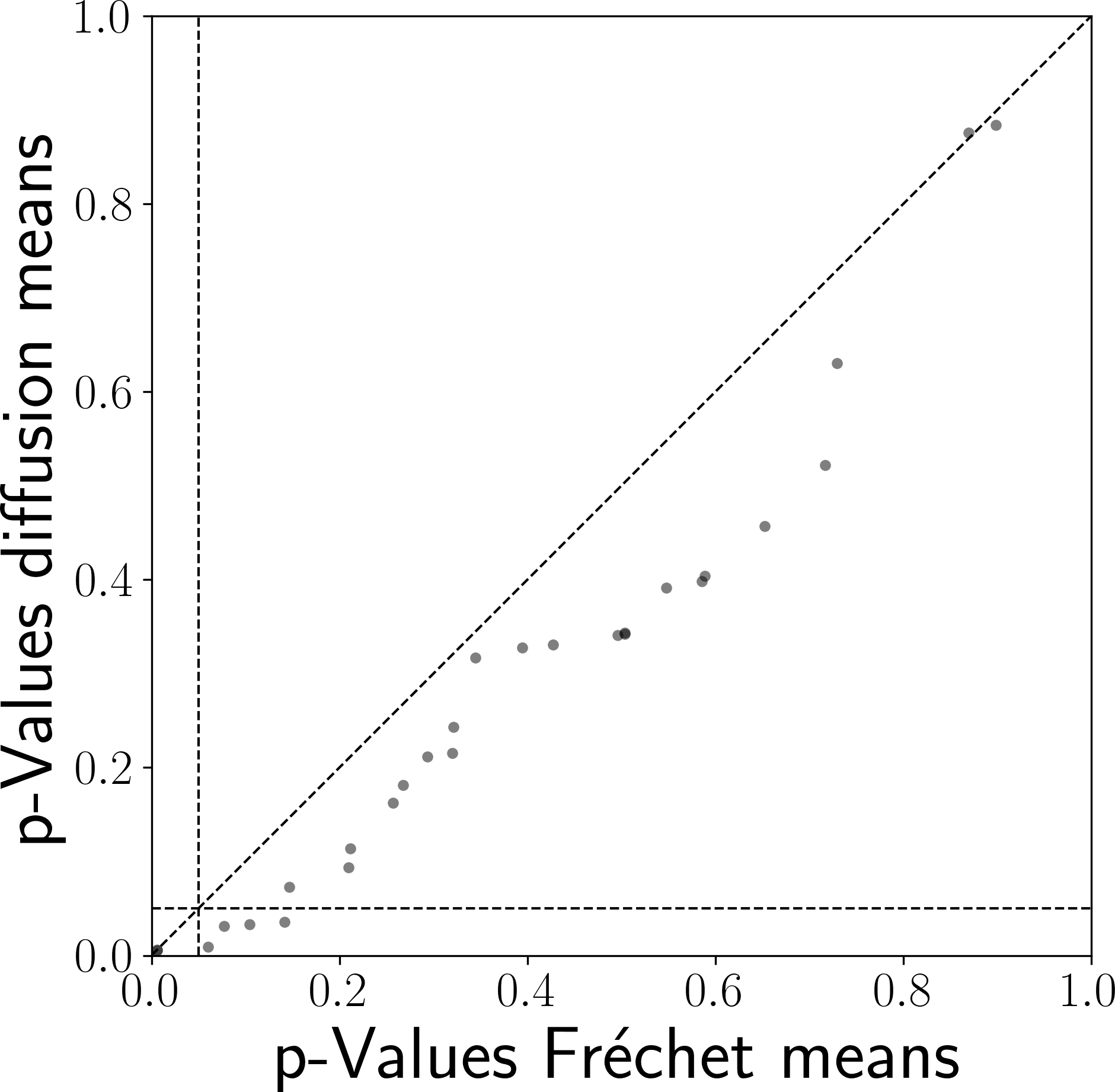} 
		\caption{Geomagnetism data}
		\label{fig:test-sub1}
	\end{subfigure}
	\hspace*{0.05\linewidth}
	\begin{subfigure}[t]{.42\linewidth}
		\centering
		\includegraphics[width=\linewidth]{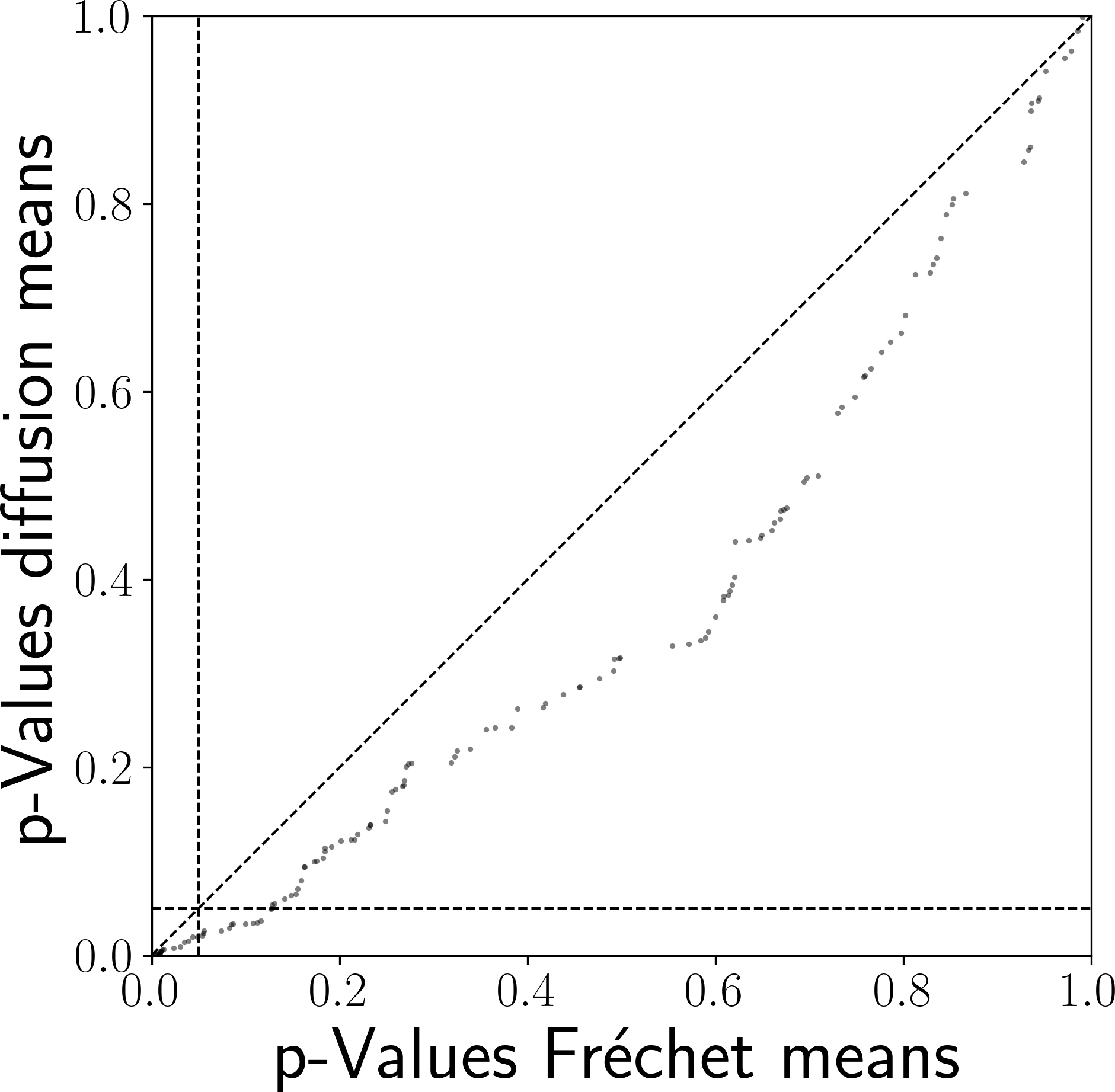}
		\caption{Wind directions}
		\label{fig:test-sub2}
	\end{subfigure}\hfill
	\caption{Two sample $T^2$ tests for pole positions in the geomagnetism data (a) and mean wind direction for yearly wind direction data sets (b). For the wind direction data, we restricted attention to tests where at least one of the data sets exhibited a variance modulation of at least $5$ and we sorted p-values to highlight the improved power.} 
	\label{fig:tests}
\end{figure}

It is immediately obvious from Figure \ref{fig:tests} that the power of the test is much higher when using diffusion means with simultaneously estimated variance, highlighting the practical usefulness of diffusion means in a realistic scenario for data on positively curved spaces.

\section*{Acknowledgments}

P. Hansen and S. Sommer gratefully acknowledge support by the Novo Nordisk Foundation grant NNF18OC0052000 and the Villum Foundation grants 22924 and 40582. B. Eltzner and S. Huckemann gratefully acknowledge funding by the DFG SFB 1456.

\appendix

\section{Auxiliary Results}

\subsection{Auxiliary Lemmas for Theorem \ref{thm:diff-ext}}\label{subsec:thm32}

\begin{lemma} \label{lem:wrapped-normal}
  The wrapped normal distribution can be expressed as
  \begin{align*}
  p(x,y,t) = \frac{1}{2\pi} \prod_{k=1}^\infty \limits \left(1-e^{-kt}\right) \left( 1 + 2 e^{-(k-1/2)t} \cos(x-y) + e^{-(2k-1)t} \right) \, .
  \end{align*}
\end{lemma}

\begin{proof}
  For any Schwartz function $f(x)$ let $g(x) := f(2\pi x +z)$. Now, using the Poisson summation formula, proved e.g. in Chapter 3.2.3 of \cite{grafakos2008}, for $g$, we have
  \begin{align*}
  \sum_{k=-\infty}^\infty \limits f(2\pi k + z) =& \sum_{k=-\infty}^\infty \limits g(k) = \sum_{k=-\infty}^\infty \limits \,\, \int_{-\infty}^{\infty} \limits g(x) e^{-i2\pi kx} dx = \sum_{k=-\infty}^\infty \limits \,\, \int_{-\infty}^{\infty} \limits f(2\pi x +z) e^{-i2\pi kx} dx\\
  =& \sum_{k=-\infty}^\infty \limits \frac{1}{2\pi} \int_{-\infty}^{\infty} \limits f(x +z) e^{-i kx} dx = \frac{1}{2\pi} \sum_{k=-\infty}^\infty \limits e^{ikz} \int_{-\infty}^{\infty} \limits f(x) e^{-i kx} dx
  \end{align*}
  
  For the wrapped normal distribution from Equation~\eqref{wrap-Gauss:eqn} in the article this yields
  \begin{align*}
  p(x,y,t) = \frac{1}{2\pi} \sum_{k=-\infty}^\infty e^{-t k^2/2 + i k (x - y)} = \frac{1}{2\pi} \sum_{k=-\infty}^\infty \left(e^{-t/2}\right)^{k^2} \left(e^{i (x - y)/2} \right)^{2k} \, .
  \end{align*}
  
  Now we can use the Jacobi Triple product (e.g. \cite{andrews1965simple})
  \begin{align*}
  \sum_{k=-\infty}^\infty a^{k^2} b^{2k} = \prod_{k=1}^\infty \left( 1 - a^{2k}\right) \left( 1 + a^{2k-1} b^2\right) \left( 1 +\frac{a^{2k-1}}{b^2}\right)
  \end{align*}
  with $a := e^{-t/2}$ and $b := e^{-i (x - y)/2}$ to get
  \begin{align*}
  p(x,y,t) = \frac{1}{2\pi} \prod_{k=1}^\infty \limits \left(1-e^{-kt}\right) \left( 1 + 2 e^{-(k-1/2)t} \cos(x-y) + e^{-(2k-1)t} \right) \, .
  \end{align*}
\end{proof}

\begin{lemma} \label{lem:rho1-derivative}
  For the function
  \begin{align*}
  \rho(x,y,s) :=& - \frac{1}{s} \sum_{k=1}^\infty \left( \log\left(1-s^{2k}\right) + \log \left( 1 + 2 s^{2k-1} \cos(x-y) + s^{4k-2}\right) \right)
  \end{align*}
  the derivative $\frac{d}{ds} \rho(x,y,s)$ exists and is uniformly bounded for any $x,y \in [-\pi, \pi)$ and $s \in [0,1/2]$.
\end{lemma}

\begin{proof}
  The infinite sum in $\rho(x,y,s)$ is absolutely convergent for $s \in [0,1/2]$ as can be seen by
  \begin{align*}
  \sum_{k=1}^\infty \left| \log\left(1-s^{2k}\right) \right| \le& \sum_{k=1}^\infty 2 \log(2) s^{2k}\\
  \sum_{k=1}^\infty \left|\log \left( 1 + 2 s^{2k-1} \cos(x-y) + s^{4k-2}\right) \right| \le& \sum_{k=1}^\infty \left( 2 s^{2k-1} + s^{4k-2} \right) \, .
  \end{align*}
  This means that derivatives by $s$ can be interchanged with the sum to yield
  \begin{align*}
  \frac{d}{ds} \rho(x,y,s) =& \frac{1}{s^2} \sum_{k=1}^\infty \left( \log\left(1-s^{2k}\right) + \log \left( 1 + 2 s^{2k-1} \cos(x-y) + s^{4k-2}\right) \right)\\
  &+ \frac{1}{s} \sum_{k=1}^\infty \left( \frac{2k s^{2k-1}}{1-s^{2k}} - \frac{(4k-2)s^{2k-2} \cos(x-y) + (4k-2) s^{4k-3}}{1 + 2 s^{2k-1} \cos(x-y) + s^{4k-2}}\right)
  \end{align*}
  The convergence of the infinite sums for $s>0$ follows from the same arguments as above. It remains to show that the limit $s \to 0$ is well-defined. Then a uniform upper bound is given simply by the maximum value. The only possibly non-zero contributions for $s \to 0$ are the summands with $k=1$
  \begin{align*}
  \lim_{s \to 0} \limits \frac{d}{ds} \rho(x,y,s) = \lim_{s \to 0} \limits & \frac{1}{s^2} \left( \log\left(1-s^{2}\right) + \log \left( 1 + 2 s \cos(x-y) + s^2\right) \right)\\
  &+ \left( \frac{2}{1-s^{2}} - \frac{2s^{-1} \cos(x-y) + 2}{1 + 2 s \cos(x-y) + s^{2}}\right) \, .
  \end{align*}
  Further, we use the expansion $\log(1+x) = x - x^2/2 + \mathcal{O}(x^3)$ to see
  \begin{align*}
  \lim_{s \to 0} \limits \frac{d}{ds} \rho(x,y,s) = \lim_{s \to 0} \limits & \left( - 1 + \mathcal{O}(s^2) + 2 s^{-1} \cos(x-y) + 1 - 2 \cos(x-y)^2 + \mathcal{O}(s) \right)\\
  &+ \left( \frac{2}{1-s^{2}} - \frac{2s^{-1} \cos(x-y) + 2}{1 + 2 s \cos(x-y) + s^{2}}\right)\\
  = \lim_{s \to 0} \limits & \left( \frac{2s^{-1} \cos(x-y) + 4 \cos(x-y)^2 + 2s \cos(x-y)}{1 + 2 s \cos(x-y) + s^{2}} - 2 \cos(x-y)^2 \right)\\
  &+ \left( \frac{2}{1-s^{2}} - \frac{2s^{-1} \cos(x-y) + 2}{1 + 2 s \cos(x-y) + s^{2}} \right) = 2 \cos(x-y)^2\, .
  \end{align*}
  The derivative $\frac{d}{ds} \rho(x,y,s)$ thus exists and is uniformly bounded for any $x,y \in [-\pi, \pi)$ and $s \in [0,1/2]$.
\end{proof}

\begin{lemma} \label{lem:rhom-derivative}
  For the function
  \begin{align*}
  \rho(x,y,s) :=& - \frac{1}{s} \log \left( \sum_{k=0}^\infty s^{k}s^{\frac{k(k-1)}{m}}\frac{2k+m-1}{m-1} C_k^{(m-1)/2}(\langle x,y\rangle_{\R^{m+1}} ) \right)
  \end{align*}
  with $m > 1$, the derivative $\frac{d}{ds} \rho(x,y,s)$ exists and is uniformly bounded for any $x,y \in [-\pi, \pi)$ and $s \in [0,1/2]$.
\end{lemma}

\begin{proof}
  For the Gegenbauer polynomials, it is known that
  \begin{align*}
  |C_k^{(m-1)/2}(x)| \le C_k^{(m-1)/2}(1) = \binom{k + m - 2}{k} \le (k+m-2)^{m-2}\, ,
  \end{align*}
  so the sum is absolutely convergent for $s \in [0, 1/2]$. Next, we calculate the derivative
  \begin{align*}
  \frac{d}{ds} \rho(x,y,s) =& \frac{1}{s^2} \log \left( \sum_{k=0}^\infty s^{k}s^{\frac{k(k-1)}{m}}\frac{2k+m-1}{m-1} C_k^{(m-1)/2}(\langle x,y\rangle_{\R^{m+1}} ) \right)\\
  &- \frac{\sum_{k=1}^\infty \limits \left(k + \frac{k(k-1)}{m} \right)s^{k-2}s^{\frac{k(k-1)}{m}}\frac{2k+m-1}{m-1} C_k^{(m-1)/2}(\langle x,y\rangle_{\R^{m+1}} )}{\sum_{k=0}^\infty \limits s^{k}s^{\frac{k(k-1)}{m}}\frac{2k+m-1}{m-1} C_k^{(m-1)/2}(\langle x,y\rangle_{\R^{m+1}} )}
  \end{align*}
  It is clear that the sums converge, so the expression exists for $s > 0$. To ensure a well-defined limit for $s \to 0$ we use the expansion $\log(1+x) = x - x^2/2 + \mathcal{O}(x^3)$ and the expression for the first Gegenbauer polynomial $C_0^{(m-1)/2} = 1$ and restrict to non-vanishing orders, i.e. $k \le 1$ in the series
  \begin{align*}
  \lim_{s \to 0} \limits \frac{d}{ds} \rho(x,y,s) = \lim_{s \to 0} \limits &\frac{1}{s^2} \log \left( 1 + s\frac{m+1}{m-1} C_1^{(m-1)/2}(\langle x,y\rangle_{\R^{m+1}} ) + \mathcal{O}(s^{2+2/m}) \right)\\
  &- \frac{s^{-1}\frac{m+1}{m-1} C_1^{(m-1)/2}(\langle x,y\rangle_{\R^{m+1}} ) + \mathcal{O}(s^{2/m})}{1 + s\frac{m+1}{m-1} C_1^{(m-1)/2}(\langle x,y\rangle_{\R^{m+1}} ) + \mathcal{O}(s^{2+2/m})}\\
  = \lim_{s \to 0} \limits & \left( s^{-1}\frac{m+1}{m-1} C_1^{(m-1)/2}(\langle x,y\rangle_{\R^{m+1}} ) + \mathcal{O}(s^{2/m}) \right)\\
  &- \frac{s^{-1}\frac{m+1}{m-1} C_1^{(m-1)/2}(\langle x,y\rangle_{\R^{m+1}} ) + \mathcal{O}(s^{2/m})}{1 + s\frac{m+1}{m-1} C_1^{(m-1)/2}(\langle x,y\rangle_{\R^{m+1}} ) + \mathcal{O}(s^{2+2/m})}\\
  =& \left( \frac{m+1}{m-1} C_1^{(m-1)/2}(\langle x,y\rangle_{\R^{m+1}} ) \right)^2
  \end{align*}
  The derivative $\frac{d}{ds} \rho(x,y,s)$ exists and is uniformly bounded for any $x,y \in S^m$ and $s \in [0,1/2]$.
\end{proof}

\subsection{Proofs of Lemmas \ref{loghdiff} and \ref{positivedifferentials}}
\label{appendix:sec3proofs}
\begin{proof}[Proof Lemma \ref{loghdiff}]
  Fix $m\geq 2$. Writing out $\frac{d}{dx} \ln \hfun(x) |_{x = y} = \frac{\hfun'(y)}{\hfun(y)}$, it follows from Remark 3 that $\frac{d}{dx}\lfun(x)$ is positive for all values of $m\geq 2$ and $t>0$. To simplify the calculations, we define the functions $\hfunl(x) := (m-1)A_{\cS}^{m}h_{t,m}(x)$ and
  \begin{align*}
  \frac{d^k}{dx^k}\hfunl^{i,j}(x) := \sum_{l=i}^j \frac{d^k}{dx^k} \Big( e^{-l(l+m-1)\frac12t}(2l+m-1)C_l^{(m-1)/2}(x) \Big).
  \end{align*}
  and for each $l\geq 0$.
  We shall complete the proof by showing that $s_{t,n}(x):=\hfunl(x)\hfunl''(x)-\hfunl'(x)^2 <0$ whenever the stated conditions on $t$ hold, by proving the two inequalities
  \begin{align}\label{firtscond}
  s_{t,n}(x)\leq & 2f_{t,m}^{0,0}(x)\frac{d^2}{dx^2}f_{t,m}^{2,2}(x)- \left(\frac34 \frac{d}{dx}f_{t,m}^{1,1}(x)\right)^2 \, ,\\
  \label{secondcond}
  B_t(x) :=& 2f_{t,m}^{0,0}(x)\frac{d^2}{dx^2}f_{t,m}^{2,2}(x)- \left(\frac34\frac{d}{dx}f_{t,m}^{1,1}(x)\right)^2 <0 \, .
  \end{align}
  In order to choose $\delta^*(m)$ such that each $t>\delta^*(m)$ satisfies Equation \eqref{secondcond}, writing out the terms yields
  \begin{equation*}
  B_t(x)
  \leq (m+1)(m-1)^2e^{-(m+1)t} \left(2(m+3)e^{-t} - \frac{9}{16}(m+1)\right)
  \end{equation*}
  which is negative for all $t>\tbound^*(m)$ when $\tbound^*(m)= \log\left(\frac{32(m+3)}{9(m+1)}\right)$. For $\tbound$ to satisfy Equation \eqref{firtscond}, we ensure that the first term of $\hfunl$ and $\frac{d^2}{dx^2}\hfunl$ dominates their respective tails. 
  
  Throughout this proof we will use the fact that $\sum_{i=0}^{\infty}a_i$ is bounded by $\frac{1}{1-r}|a_0|$ if for all $n\in \N_0$ it holds that $\left|\frac{a_{n+1}}{a_n}\right|<r$ for $r\in (0,1)$. We seek $\tbound_1(m)>0$ such that for all $t>\tbound_1(m)$, the bounds $|\hfunl^{0,\infty}(x)|\leq \sqrt{2}\hfunl^{0,0}(x)$ and $|\frac{d^2}{dx^2}\hfunl^{2,\infty}(x)|\leq \sqrt{2}\frac{d^2}{dx^2}\hfunl^{2,2}(x)$ hold. Using the identity $|C_n^\alpha(x)|\leq C_n^\alpha(1)$, it suffices to prove bounds for the series at $x=1$. Writing out the terms and using the identities and $C_n^\alpha(1) = \frac{\Gamma(n+2\alpha)}{n!\Gamma(2\alpha)}$ and $\frac{d}{dx}C_n^\alpha(x) = 2\alpha C_{n-1}^\alpha(x)$, we obtain the bounds for all $n\geq l$, 
  \begin{align}\label{boundquo}
  \left|\frac{\frac{d}{dx}\hfunl^{n+1,n+1}(1)}{\frac{d}{dx}\hfunl^{n,n}(1)}\right| \leq \frac{(m+n+l)}{n+1-l}e^{-(m+n)\frac12t} \leq (2l + m)e^{-(m+l)t}\leq me^{-\frac12mt}
  \end{align}
  To obtain the two bounds, choose $\tbound_1(m) = -\frac2m\ln\left(\frac{\sqrt{2}-1}{\sqrt{2}m} \right)$. 
  The final step in proving the bound in Equation \eqref{firtscond}, is to establish the upper bound $|\frac{d}{dx}\hfunl^{2,\infty}(x)| \leq \frac14\frac{d}{dx} \hfunl^{1,1}(x)$, since then 
  \begin{equation*}
  \frac{d}{dx}\hfunl(x)^2 \geq \left(\frac{d}{dx}\hfunl^{1,1}(x)- \frac14\frac{d}{dx}\hfunl^{1,1}(x)\right)^2 = \left( \frac34\frac{d}{dx}\hfunl^{1,1}(x)\right)^2
  \end{equation*}
  Notice that $|\frac{d}{dx}\hfunl^{2,\infty}(x)| \leq \frac{d}{dx}\hfunl^{2,\infty}(1)$ and $\frac14\frac{d}{dx} \hfunl^{1,1}(x)=\frac{1}{4}(m+1)e^{-(m-1)\frac12t}$. So in order to obtain the upper bound, we solve 
  \begin{equation}\label{first}
  \left| \frac{\frac{d}{dx}\hfunl^{3,3}(1)}{\frac{d}{dx}\hfunl^{2,2}(1)} \right| \leq 1-4(m+3)e^{-(m+1)\frac12t}
  \end{equation}
  Using the bound in Equation \eqref{boundquo}, we find the bound $\delta_2(m)= -\frac{2}{(m+1)} \ln \left(\frac{1}{(m+1)+4(m+3)} \right)$ by solving $(m+1)e^{-(m+1)\frac12t} = 1-4(m+3)e^{-(m+1)\frac12t}$.
  Since $\delta_1(m)<\delta_2(m)$ for all $m\geq 2$, we choose 
  \begin{align*}
  \delta(m) &= \max\{\delta^*(m),\delta_2(m)\}
  \end{align*}
\end{proof}

\begin{proof}[Proof Lemma \ref{positivedifferentials}]
  The first claim follows immediately from the definition of $\hfun$ and Remark 3. For the second claim, see that 
  \begin{equation*}
  \frac{d^n}{dx^n} \hfun(x) = \sum_{l=n}^{\infty} e^{-l(l+m-1)\frac12t}\frac{2l+m-1}{m-1} \frac{1}{A_{\cS}^{m}} ((m-1)\cdot...\cdot (m+2n-1))C_{l-n}^{(m-1)/2+n}(x).
  \end{equation*}
  where the first term is positive since the $C_0^{\frac{(m-1)}{2}+n}(x)=1$. To ensure that a sum $\sum_{i=0}^\infty a_i$ with $a_0>0$ is positive, it suffices to prove that $a_1>2^l|a_l|$ for all $l\geq1$. In our case, this means proving
  \begin{align*}
  e^{-n(n+m-1)\frac12t}(2n+m-1) \geq 2^l e^{-(n+l)(n+l+m-1)\frac12t}(2(n+l)+m-1)|C_{l-n}^{(m-1)/2+n}(x)|
  \end{align*}
  for all $l\geq 1$. Using the bound $|C_{l}^{(m-1)/2+n}(x)| \leq \frac{(m-1+2n)_l}{l!}$ \cite{kiepiela_gegenbauer_2003}, it suffices to show
  \begin{equation}\label{boundcoef}
  (2n+m-1) \geq 2^le^{-(l^2+2ln+lm-l)\frac12t}(2(n+l)+m-1)\frac{(m-1+2n)_l}{l!}
  \end{equation}
  where the (RHS) is a decreasing function in both $l,n$ and $m$ whenever $t\geq 1$. The inequality holds for $m=2$, $n=2$ and $l=1$, which completes the proof. 
\end{proof}

\subsection{Proof of Theorem \ref{thm:diff-mean-frechet-mean}}\label{subsec:thm47}

\begin{proof}
  
  First, we establish the following three claims where the first claim includes the first assertion of the theorem.
  \begin{description}
    \item[Claim 1:] $M \neq \emptyset$ is closed and bounded.
    \item[Claim 2:] $F_{t_k}(y_k) \to F_0(y)$ whenever $y_k \to y$ and $t_k \to 0$.
    \item [Claim 3:] If $E_{t}\neq \emptyset$ for all $t > 0$ sufficiently small, then $\min_{y\in \cM} F_{t}(y) =: \lw_{t} \to \lw_0 := \min_{y\in \cM} F_0(y)$ as $t\to 0$.
  \end{description}
  
  To see the first claim, note that local tightness, Equation~\eqref{eq:tightness}, guarantees the existence of $T(y) >0$ for every $y\in \cM$ such that $F_{t}(y) < \infty$ for all $t< T(y)$. For $t=0$, invoking the triangle inequality, note that $\dist(x,y)^2 \leq 2 \dist(X,y)^2 + 2\dist(X,x)^2$. Taking expected values yields $F_0(x) \geq \frac{1}{2}\dist(x,y)^2 - F_0(y)$. In consequence, $F_0$ assumes a minimum over the complete space $\cM$. By the same argument the set $M$ of minima is bounded, and as the preimage of a point under a continuous function, it is closed.
  
  To see the second claim, let $\epsilon > 0$, $\cM \ni y_k \to y\in \cM$ and let $K_1 \subseteq K_2 \subseteq \ldots \cM$ be an exhaustion by compact sets. Then Equation~\eqref{eq:uniform-conv} assures that for every $j \in \N$ there is $T_j>0$ such that
  $$ \int_{K_j} |\rho_{t_k}(X,y_k) - \rho_0(X,y_k)|\,d\bP^X \leq \frac{\epsilon}{4}\mbox{ whenever } t_k < T_j\,.$$
  By continuity over a compact set we also find $k_j \in \N$ such that
  $$ \int_{K_j} |\rho_{0}(X,y_k) - \rho_0(X,y)|\,d\bP^X \leq \frac{\epsilon}{4}\mbox{ whenever } k > k_j\,.$$
  Hence, whenever $k > k_j$ and $t_k<T_j$
  \begin{eqnarray*} \lefteqn{
      \int_{K_j} |\rho_{t_k}(X,y_k) - \rho_0(X,y)|\,d\bP^X}\\ 
    &\leq &
    \int_{K_j} |\rho_{t_k}(X,y_k) - \rho_0(X,y_k)|\,d\bP^X + 
    \int_{K_j} |\rho_{0}(X,y_k) - \rho_0(X,y)|\,d\bP^X \leq \frac{\epsilon}{2}\,.
  \end{eqnarray*}
  On the other hand, by the tightness condition, Equation~\eqref{eq:tightness}, 
  $$ \int_{\cM\setminus K_j}|\rho_{t_k}(X,y_k) - \rho_0(X,y)|\,d\bP^X \leq \int_{\cM \setminus K_j}\rho_{t_k}(X,y_k)\,d\bP^X +\int_{\cM\setminus K_j}\rho_0(X,y)\,d\bP^X \leq \frac{\epsilon}{2}$$
  whenever $j\geq j(\epsilon/4)$ and $0\leq t_k \leq T_0$. In consequence, setting $T^* = \min\{T(y),T_{j(\epsilon/4)}\}$, we have that
  $$ |F_{t_k}(y_k) - F_0(y)| \leq \int_{K_{j(\epsilon/4)}\cup (\cM \setminus K_{j(\epsilon/4)})} |\rho_t(X,y) - \rho_0(X,y)|\,d\bP^X \leq \epsilon $$
  for all $0\leq t_k < T^*$, yielding the first claim.
  
  To see the third claim, pick $y_t \in \cM$ with $F_{t}(y_{t}) =\lw_{t}$ for all $t\geq 0$ sufficiently small. Then, by Claim 2, for every $\epsilon >0$ there is $T(\epsilon) >0$ such that 
  $$\lw_0\leq F_0(y_{t}) \leq F_{t}(y_{t}) - \epsilon = \lw_{t} - \epsilon$$
  whenever $0\leq t \leq T(\epsilon)$, yielding $\lw_0 \leq \lim\inf_{t\to 0} \lw_t$.
  Similarly, for every $\epsilon >0$ there is $T'(\epsilon) >0$ such that
  $$\lw_0\geq F_t(y_0) - \epsilon \geq \lw_t - \epsilon$$
  for $0\leq t < T'(\epsilon)$, yielding $\lw_0 \geq \lim\sup_{t\to 0} \lw_t$. This establishes the third claim.
  
  Now we show the second assertion of the Theorem, namely Ziezold convergence. It is trivial if $E_{t_k} = \emptyset$ for all $k\in \N$ sufficiently large. Passing to a sub-sequence if necessary, we may thus assume that $E_{t_k} \neq \emptyset$ for all $k \in \N$. Then every $y\in \mathop{\bigcap}_{i=1}^\infty \overline{\mathop{\bigcup}_{k=1}^\infty E_{t_k}}$ is a cluster point of a sequence $y_k \in E_{t_k}$ with
  $$F_{t_k}(y_k) = \lw_{t_k} \to \lw_0$$
  by the third claim. By Claim 2 we have thus $F_{t_k}(y_k)\to F_0(y) = \lw_0$ yielding $y\in M$. 
  
  Finally, we show the third assertion of the Theorem, namely Bhattacharya and Patrangenaru convergence. It is trivial if $E_{t_k} = \emptyset$ for all $k\in \N$ sufficiently large. As reasoned above, we may thus assume $E_{t_k} \neq \emptyset$ for all $k \in \N$.
  Then $r_k := \sup_{y\in E_{t_k}} \dist(y,M)$ where $\dist(y,M) := \min_{x\in M} \dist(y,x)$ is well defined because the minimum is assumed as $M$ is a closed set by Claim 2 and $\cM$ is complete by hypothesis. 
  If $r_k$ was unbounded, we might assume, if necessary passing to a sub-sequence, that $r_k \to \infty$.
  In consequence, there would be sequences $y'_k \in E_{t_k}$ and $x_k \in M$ with $\dist(y'_k,x_k) \to \infty$. In consequence of coercivity (Hypothesis 2 of the theorem, among others, asserting suitable $y_0\in \cM$ and $C,R_k>0$), since $M$ is bounded by Claim 1, then $d(y'_k,y_0)\to \infty$ and hence
  \begin{eqnarray*}
    \lw_{t_k} &=& \E[\rho_t(X,y'_k)] \geq \int_{\{x\in \cM: \rho_t(x,y_0)<C\}} \rho_t(X,y'_k)\,d\bP^X \geq R_k \bP\{\rho_t(X,y_0)<C\} \to \infty
  \end{eqnarray*}
  in contradiction to $\lw_{t_k} \to \lw_0$ by Claim 3. Hence, $r_k$ is bounded and, since $M$ is bounded by Claim 1, every sequence $y'_k \in E_{t_k}$ is bounded as well. By Heine-Borel (Hypothesis 1. of the theorem), every such sequence has a cluster point $y$ which, by the second assertion of the theorem, is in $M$, i.e. $r_k \to 0$. This yields Bhattacharya and Patrangenaru convergence.
\end{proof}

\section{Example: point mass and uniform hemisphere}\label{sec:appendix-example}
Let $X$ be a random point on $\cS^2$ with the following distribution: It is uniformly distributed on the lower hemisphere $\mathbb{L}=\{(q_1,q_2,q_3)\in \cS^2 | q_2 \leq 0 \}$ with total mass $0<\alpha <1$ and it has one point mass at the north pole $\mu = (0,1,0)$ with probability $1-\alpha$. Writing $\ell_t(x)=\ell_{t,2}(x)$, the log-likelihood function for $t>0$ and $y_\delta\in \cS^2$ is
\begin{align*}
L_{t}(y_\delta)&=-(1-\alpha)\ln(p(\mu,y_\delta,t))- \frac{\alpha}{\text{vol}(\mathbb{L})} \int_{\mathbb{L}} \ln(p(x,y_\delta,t))dx \\
&= -(1-\alpha)\ell_t(\langle\mu,y_\delta\rangle_{\R^3})- \frac{\alpha}{\text{vol}(\mathbb{L})} \int_{\mathbb{L}} \ell_t(\langle x,y_\delta\rangle_{\R^3})dx
\end{align*}
The log-likelihood function only depends on $\delta$, thus we can consider it as a function $\widetilde{L}_t : \delta \to L_t(y_\delta)$ and using the method of \cite{eltzner_smeary_2018} we can write the function in terms of crescent integrals
\begin{equation*}
\widetilde{L}_t(\delta)=\widetilde{L}_t(0) + \alpha \Big(\frac{1}{2\pi}\Big)\big(C_+(\delta)-C_-(\delta) \big)- (1-\alpha)\ell_t(\cos\delta)+(1-\alpha)\ell_t(1)
\end{equation*}
where the crescent integrals $C_+(\delta)$ and $C_-(\delta)$ are defined by
\begin{align*}
C_+(\delta) &= 
-\int_{-\pi/2}^{\pi/2} \cos\theta\int_{0}^{\delta} \ell_t(\cos\theta\sin\phi)d\phi d\theta \\
C_-(\delta) &= 
-\int_{-\pi/2}^{\pi/2} \cos\theta\int_{0}^{\delta} \ell_t(-\cos\theta\sin\phi)d\phi d\theta.
\end{align*}
In order to determine the diffusion $t$-means, we begin by calculating the first derivative of $\widetilde{L}_{t}$. 
The derivative of the crescent integral term is
\begin{align*}
\frac{d}{d\delta}(C_+(\delta)-C_-(\delta)) &= \int_{-\pi/2}^{\pi/2} \cos\theta\frac{d}{d\delta} \int_{0}^{\delta}-\ell_t(\cos\theta\sin\phi)+ \ell_t(-\cos\theta\sin\phi) d\phi d\theta
\\&= \int_{-\pi/2}^{\pi/2}\cos\theta (-\ell_t(\cos\theta\sin\delta)+ \ell_t(-\cos\theta\sin\delta)) d \theta
\\ 
&= - \sin\delta\int_{-\pi/2}^{\pi/2} \sin\theta^2(\ell'_t(\cos\theta\sin\delta)+ \ell'_t(-\cos\theta\sin\delta)) d \theta.
\end{align*}
Using the substitution $\theta = \arccos(x)$ we obtain
\begin{align*}
\frac{d}{d\delta}(C_+(\delta)-C_-(\delta)) &= -2\sin\delta\int_{0}^{1} (1-x^2) (\ell'_t(x\sin\delta)+ \ell'_t(-x\sin\delta)) (\arccos)'(x) dx \\
&= -2\sin\delta \int_{0}^{1} \sqrt{1-x^2} (\ell'_t(x\sin\delta)+ \ell'_t(-x\sin\delta)) dx.
\end{align*}

\begin{lemma}\label{convexlemma}
  The function $\ell'_{t,2}$ is convex on $[-1,1]$ for $t>2.12$. 
\end{lemma}
\begin{proof}
  Adopting the notation in the proof of Lemma \ref{loghdiff} and writing $h_t=h_{t,2}$ and $f_t = f_{t,2}$, the second derivative of $\ell'_{t,2}$ is 
  \begin{equation*}
  \frac{d^3}{dx^3}\ell_{t,2} = \frac{h_t'''(x)h_t(x)-h_t''(x)h'_t(x)}{h(x)^2} - 2 \frac{h_t''(x)h_t(x)-h_t'(x)^2}{h(x)^2}
  \end{equation*}
  so we have completed the proof if we show that for all $t>1.1$ it holds that 
  \begin{equation*}
  f_t'''(x)f_t(x)-f_t''(x)f'_t(x)- 2 (f_t''(x)f_t(x)-f_t'(x)^2) >0. 
  \end{equation*}
  In the proof of Lemma \ref{loghdiff}, we showed an upper bound for $(f_t''(x)f_t(x)-f_t'(x)^2)$ if $t>\delta(2)$. Using the strategy from the proof of Lemma \ref{positivedifferentials}, we can ensure the bound $f_t'''(x)f_t(x) > \frac{1}{4}\frac{d^3}{dx^3}f_{t,2}^{3,3}(x)f_{t,2}^{0,0}(x)$. For this to hold, we must find a limit on $t$ such that $\frac{d^i}{dx^i}f_{t,2}^{i,i}(x)> 3^l|\frac{d^i}{dx^i}f_{t,2}^{i+l,l}(x)|$ holds for $i=0,3$ and all $l\geq 1$ and $x\in [-1,1]$. This is done by solving the inequality 
  \begin{equation*}
  (2n+1) \geq 3^le^{-(l^2+2ln+l)\frac12t}(2(n+l)+1)\frac{(1+2n)_l}{l!}
  \end{equation*}
  for all $l\geq 1$.
  Solving for a bound on $t$, we see that the inequality holds for all $l\geq 1$ when $t>2.12$. Due to the generous bound on $t$, we can improve the bounds obtained in the proof of Lemma \ref{loghdiff} and obtain $|f_t''(x)f'_t(x)| < \frac{9}{4}\frac{d^2}{dx^2}f_{t,2}^{2,}(x)\frac{d}{dx}f_{t,2}^{1,1}(x)$. So we have reduced the problem to determining the bound on $t$ such that 
  \begin{equation*}
  \frac{1}{4}\frac{d^3}{dx^3}f_{t,2}^{3,3}(x)f_{t,2}^{0,0}(x)- \frac{9}{4}\frac{d^2}{dx^2}f_{t,2}^{2,}(x)\frac{d}{dx}f_{t,2}^{1,1}(x)-
  2\left(2f_{t,m}^{0,0}(x)\frac{d^2}{dx^2}f_{t,m}^{2,2}(x)- \left(\frac34 \frac{d}{dx}f_{t,m}^{1,1}(x)\right)^2\right)
  \end{equation*}
  is positive. Writing out the terms, this is equivalent to
  \begin{align*}
  e^{-(2m+1)\frac12t}(m+3)- (m+3)e^{-(m+1)\frac12t}(9(m+1)+16) +3(m+1)>0
  \end{align*}
  which holds for all $t>2.12$. 
\end{proof}

By Lemma \ref{convexlemma} the map $\ell_t'$ is convex for $t>2.12$, which implies that $2\ell'_t(0) \leq \ell_t'(x) +\ell_t'(-x)$ holds for all $x\in [-1,1]$. This means that the integral can be bounded from below 
\begin{equation*}
\frac{d}{d\delta}(C_+(\delta)-C_-(\delta)) \geq -2\sin\delta \int_{0}^{1} 2\sqrt{1-x^2} \ell'_t(0) dx = -\pi\sin\delta\ell_t(0).
\end{equation*}
\begin{figure}[ht]
  \centering
  \begin{subfigure}{.5\textwidth}
    \centering
    \includegraphics[height =4.55cm, trim={1 0 0 0},clip]{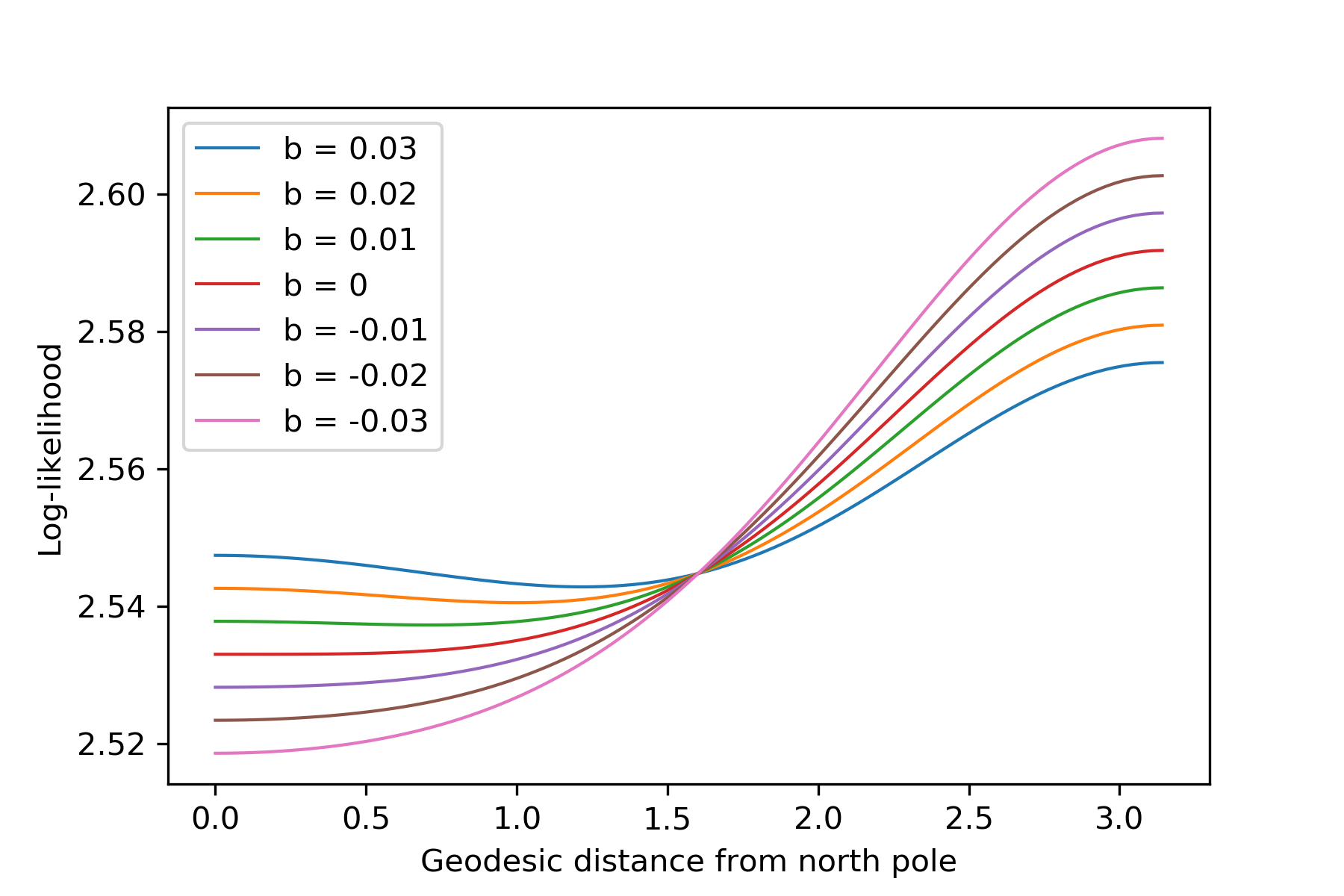}
    \caption{The negative log-likelihood function for \\ $d=2$, $t=2.2$ and $\alpha = \Sigma(t)+b$.}
    \label{fig:sub1}
  \end{subfigure}\hfill
  \begin{subfigure}{.5\textwidth}
    \centering
    \includegraphics[height =4.5cm, trim={0 0 0 0},clip]{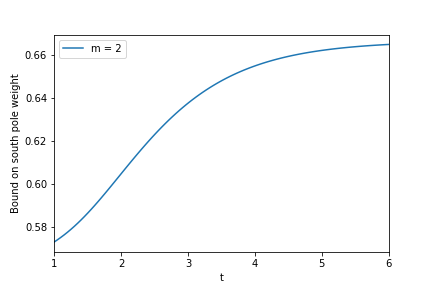}
    \caption{The critical values $\Sigma(t)$ for which the diffusion $t$-mean is $\mu$ on the $2$-sphere. }
    \label{1PoleUnifPlot}
  \end{subfigure}\hfill
  \caption{The distribution of Example 7 has a unique diffusion $t$-mean when $\alpha\leq\Sigma(t)$ for $t\geq 2.12$ and the estimator is 2-smeary at the critical value $\Sigma(t)$ plotted in panel~(b). In panel~(a) we see the log-likelihood function for different values of $\alpha$ around the critical value.}
\end{figure} 
Finally, the crescent integral bound can be used to bound the derivative
\begin{align*}
\widetilde{L}_t'(\delta) 
&\geq \sin\delta \left((1-\alpha)\frac{h_t'(\cos\delta)}{h_t(\cos\delta)}-\alpha \frac{h_t'(0)}{2h_t(0)} \right).
\end{align*}
Since $2.12>\delta(2)$, it follows by Lemma \ref{loghdiff} that the lower bound is increasing in $\delta$. This implies that the derivative $\widetilde{L}'_{\alpha,t}(\delta)$ is positive on $(0,\pi)$ whenever $\alpha \leq \Sigma(t)$ where
\begin{equation*}
\Sigma(t) = \frac{2h_t'(1)h_t(0)}{h'_t(0)h_t(1) + 2h'_t(1)h_t(0)}. 
\end{equation*}
It follows directly that $\mu$ is the unique diffusion $t$-mean when $t>2.12$ and $\alpha \leq \Sigma(t)$. Furthermore, the second derivative at $\delta =0$
\begin{equation*}
\widetilde{L}''_t(\delta) = (1-\alpha)\frac{h_t'(1)}{h_t(1)}-\alpha \frac{h_t'(0)}{2h_t(0)} 
\end{equation*}
vanishes if and only if $\alpha = \Sigma(t)$. By Definition 4.6 the estimator is thus at least $2$-smeary for this $\alpha = \Sigma(t)$ and that the regular CLT rate holds for $\alpha < \Sigma(t)$. 

\bibliographystyle{Chicago}
\bibliography{BibPaper}

\end{document}